\documentclass[12pt]{article}
\usepackage{amssymb}
\usepackage{color, amsmath,amssymb, amsfonts, amstext,amsthm, latexsym}

\usepackage{amssymb, epsfig, amssymb, latexsym}
\usepackage{amsmath}
\usepackage{graphicx}
\usepackage{longtable}
\usepackage{graphicx}
\usepackage{subfigure}

\usepackage[colorlinks=false, linktocpage=true]{hyperref}

\numberwithin{equation}{section} \topmargin -0.4in

\renewcommand{\phi}{\varphi}

\newcommand{\R}{{\mathbb R}}

\newcommand{\eps}{\varepsilon}

\newtheorem{theorem}{Theorem}
\newtheorem{lemma}{Lemma}

\title{Approximation of Random Slow Manifolds and Settling of Inertial Particles under Uncertainty
\footnote{This work was partially supported by  the NSF grant
1025422, the  NSFC grants 11271290 and 11271013,
and the Office of Naval Research  under the grant N00014-12-1-0257.}}

\author{Jian Ren$^{1}$, Jinqiao Duan$^{1, 2}$ and Christopher K. R. T. Jones$^3$\\1. School of Mathematics and Statistics\\
Huazhong University of Science and Technology, Wuhan 430074,
China\\E-mail: renjian0371@gmail.com
\\2. Institute for Pure and Applied Mathematics, UCLA\\ Los Angeles, CA 90095, USA\\Email: jduan@ipam.ucla.edu\\ and\\
  Department of Applied Mathematics \\Illinois Institute of Technology
\\Chicago, IL 60616, USA\\Email: duan@iit.edu\\
3. Department of Mathematics\\
University of North Carolina-Chapel Hill\\ Chapel Hill, N.C.
27599-3250\\E-mail: ckrtj@amath.unc.edu }
\begin{document}
\date{\today}

\maketitle

\pagestyle{plain}

\begin{abstract}
A method is provided for approximating random slow manifolds of a
class of slow-fast stochastic dynamical systems. Thus  approximate,
low dimensional, reduced  slow systems are obtained  analytically in
the case of sufficiently large time scale separation. To illustrate
this dimension reduction procedure, the impact of random
environmental fluctuations on the settling motion  of inertial
particles in a cellular  flow field is examined. It is found that
noise   delays settling for some particles but enhances settling for
others. A deterministic stable manifold is an agent to facilitate
this phenomenon. Overall, noise appears to delay the settling in an
averaged sense.



\medskip



\end{abstract}

\section{Introduction}  \label{intro}

Complex dynamical systems in science and engineering often involve
multiple time scales, such as slow and fast time scales, as well as uncertainty caused by noisy fluctuations.
 For example, aerosol and pollutant particles, occur in various natural contexts
(e.g., in atmosphere and ocean coasts \cite{Beven, Clark, Haller})
and engineering systems (e.g. spray droplets), are described by
coupled system of differential equations. Some particles move fast
while others move slower, and they are usually subject to random
influences, due to molecular diffusion, environmental fluctuations,
or other small scale mechanisms that are not explicitly modeled
\cite{Schnoor}. Invariant manifolds are geometric structures in
state space that help describe dynamical behaviors of dynamical
systems.  A slow manifold is a special invariant manifold, with an
exponential attracting property and with the dimension the same as
the number of slow variables. The reduced system on a slow manifold thus
characterizes the long time dynamics in a lower dimensional setting, facilitating geometric and numerical investigation.


Existence for slow manifolds  of stochastic dynamical systems with slow-fast time scales  has been investigated recently \cite{Schm, FuLiuDuan}.
However, stochastic slow manifolds   are difficult to depict or visualize.
Therefore, in this paper, we approximate these random geometric invariant structures in the case of large time scale separation.
We derive an asymptotic approximation for these stochastic manifolds,
and illustrate the random slow manifold reduction by considering the motion of aerosol particles in a random cellular fluid flow.
The reduced slow system, being lower dimensional, facilitates our understanding of particle settling.

We comment that  approximations for individual solution paths (not a
stochastic slow manifold) for stochastic slow-fast systems have been
well investigated \cite{Freidlin, Gentz, Kabanov}. Approximations
for deterministic slow manifolds have also been considered
\cite{Jones, Rubin}.

This paper is organized as follows.
An approximation method for random slow manifolds is considered in \S \ref{slow888},
and the dynamics of aerosol particles in a random flow field is investigated in \S \ref{settle888}.

\section{Approximating random  slow manifolds and dimension reduction}
\label{slow888}

We first examine the  existence of a random slow manifold for a
slow-fast stochastic dynamical system,   then devise an
approximation method for this slow manifold, and thus obtain a low
dimensional, reduced system for the evolution of slow dynamics.

 We  consider the following slow-fast system  of stochastic differential equations (SDEs)
\begin{eqnarray}\label{SDE}
\begin{cases}
\dot{x} = Ax + f(x, y), \,\,\,\,\,   x\in \mathbb{R}^n,\\
\dot{y} = \frac{1}{\varepsilon} B y + \frac{1}{\varepsilon}g(x, y) +
\frac{\sigma}{\sqrt{\varepsilon}}\dot{W_t},\,\, y\in \mathbb{R}^m.
\end{cases}
\end{eqnarray}
Here  $A$ and $B$ are respectively $n\times n$ and $m\times m$
matrices.   The nonlinear functions
$~f:\mathbb{R}^n\times\mathbb{R}^m\rightarrow\mathbb{R}^n$  and
$~g:\mathbb{R}^n\times\mathbb{R}^m\rightarrow\mathbb{R}^m$  are
$C^1$-smooth and Lipschitz continuous with Lipschitz constants $L_f$
and $L_g$, respectively.  The parameter $\sigma$ is a positive
number and the parameter $\varepsilon>0$ is small (representing
scale separation). The stochastic process $\{W_t:t\in\mathbb{R}\}$
is a two-sided $\mathbb{R}^m$-valued Wiener process. When $f$ and
$g$ are locally Lipschitz but the system has a bounded (i.e., in
mean square norm) absorbing set, a useful trick is to cut-off  the
nonlinearities to zero outside the absorbing set, so that the new
system has global Lipscitz nonlinearities and   has the same long
time random dynamics as the original system. The existence of a
random slow manifold  for this system has been considered in
\cite{Schm} but we adopt a method from our earlier work
\cite{FuLiuDuan}.

We recall the definition of a random dynamical system (RDS) in
  a probability space $(\Omega, \mathcal{F}, \mathbb{P})$. Let
  $\theta=\{\theta_t\}_{t\in\mathbb{R}}$ be a $\mathcal{B}(\mathbb{R})\otimes\mathcal{F}$,
$\mathcal{F}-$measurable flow, i.e.,
$$\theta: \mathbb{R}\times\Omega\rightarrow\Omega, \quad
\theta_0=id_\Omega, \quad
\theta_{t_1}\circ\theta_{t_2}=\theta_{t_1+t_2},\quad \text{for}\;
t_1, t_2\in \mathbb{R}.$$ Additionally, the measure $\mathbb{P}$ is
supposed to be an invariant measure for $\theta_t$, i.e.
$\theta_t\mathbb{P}=\mathbb{P}$, for all $t\in \mathbb{R}$. For a
Wiener process driving   system, we take $\Omega=C_0(\mathbb{R},
\mathbb{R}^m)$ consisting of all continuous sample paths of
$\omega(t)$ on $\mathbb{R}$ with values in $\mathbb{R}^m$ and
$\omega(0)=0$. On $\Omega$, the flow $\theta_t$ is given by the
Wiener shift
$$\theta_t\omega(\cdot)=\omega(\cdot+t)-\omega(t),\;
\omega\in\Omega,\;\; t\in\mathbb{R}.$$
A measurable map $\varphi:
\mathbb{R}^+\times\Omega\times\mathbb{R}^{n+m}\rightarrow\mathbb{R}^{n+m}$
is said to satisfy the cocycle property  if
$$\varphi(0, \omega, x)=x,\; \varphi(t+s, \omega, x)=\varphi(t, \theta_s\omega, \varphi(s, \omega,
x)),
\quad \text{for}\; s, t\in \mathbb{R}^+,\; \omega\in\Omega\;
\text{and}\; x\in\mathbb{R}^{n+m}.$$
A random dynamical system consists of a driving system $\theta$ and
a measurable map with the cocycle property.

 Introduce a Banach Space $C_\lambda$ as our working space for random slow manifolds. For   $\lambda > 0$,    
 define  \\
$$C_\lambda ^ {1} = \big\{\nu:(-\infty, 0]\rightarrow R^n
 :\nu\; \text{ is continuous and } \sup\limits_{t\leq 0} |e^{\lambda t}\nu(t)|_{R^n}<
\infty\big\},$$and$$C_\lambda ^ {2} = \big\{\nu:(-\infty,
0]\rightarrow R^m
 :\nu\; \text{ is continuous and } \sup\limits_{t\leq 0} |e^{\lambda t}\nu(t)|_{R^m}<
\infty\big\},$$\\with the following norms respectively
\[|\nu(t)|_{C_\lambda^1} = \sup\limits_{t\leq 0} |e^{\lambda t}\nu(t)|_{R^n},
 \] and \[|\nu(t)|_{C_\lambda^2} = \sup\limits_{t\leq 0} |e^{\lambda t}\nu(t)|_{R^m}. \]
 Let $C_\lambda$ be the product Banach space $C_\lambda:=C_\lambda^1\times
 C_\lambda^2$, with the norm
 $$|(X, Y)|_{C_\lambda} = |X|_{C_\lambda^1} + |Y|_{C_\lambda^2}.$$

 For matrices $A$ and $B$, we make the following assumptions:

$\mathbf{H1}$: There are  constants $\alpha$, $\beta$  and $K$,
satisfying $-\beta<0\leq\alpha$ and $K>0$, such that for every $x\in
R^n$ and $y\in R^m$, the following exponential estimates hold:
 $$|e^{At}x|_{R^n} \leq K e^{\alpha
t}|x|_{R^n},\quad t\leq 0;\quad \quad |e^{Bt}y|_{R^m} \leq K
e^{-\beta t}|y|_{R^m},\quad t \geq 0.$$

$\mathbf{H2}$: $\beta>K L_g$.\\

In order to use the random invariant manifold framework \cite{Duan},
we transfer an SDE system into a random differential equation (RDE)
system. Introduce the following linear Langevin system
\begin{eqnarray}\label{linear-var}
   dy = \frac{B}{\varepsilon}y dt +
   \frac{\sigma}{\sqrt{\varepsilon}}dW_t.
\end{eqnarray}
It is   known \cite{Duan} that the following  process $\eta_\sigma^\varepsilon(\omega)$ is
the stationary solution of the linear system \eqref{linear-var}
$$\eta_\sigma^\varepsilon(\omega)=\frac{\sigma}{\sqrt{\varepsilon}}\int_{-\infty}^0
   e^{-\frac{B}{\varepsilon}s}\,dW_s\triangleq\sigma \eta^\varepsilon(\omega).$$
   Moreover,
$$\eta_\sigma^\varepsilon(\theta_t\omega)=\frac{\sigma}{\sqrt{\varepsilon}}\int_{-\infty}^t
   e^{\frac{B}{\varepsilon}(t-s)}\,dW_s\triangleq\sigma \eta^\varepsilon(\theta_t\omega).$$
Similarly, $\eta_\sigma(\omega)$ is the stationary solution of the
following linear SDE system
\begin{eqnarray}\label{linear}
 dy = By dt + \sigma \,dW_t,
\end{eqnarray}
with
$$\eta_\sigma(\omega)=\sigma\int_{-\infty}^0
   e^{-Bs}\,dW_s\triangleq\sigma\eta(\omega),$$ and
 $$\eta_\sigma(\theta_t\omega)=\sigma\int_{-\infty}^t
   e^{B(t-s)}\,dW_s\triangleq\sigma\eta(\theta_t\omega).$$
 Denoting
$~W_t(\psi_\varepsilon\omega)\triangleq
\frac{1}{\sqrt{\varepsilon}}W_{t\varepsilon}(\omega)$, which is also
a Wiener Process \cite{Wang}, and  has the same distribution as
$W_t(\omega)$, with $\psi_\varepsilon:\Omega\rightarrow\Omega$ .
Therefore, by a transformation $s'=s/\varepsilon$ at the second
equal sign and then omitting the prime in $s'$, we have
\begin{eqnarray}\label{linear-tvar-solu}
\eta^\varepsilon(\theta_{t\varepsilon}\omega) =
   \frac{1}{\sqrt{\varepsilon}}\int_{-\infty}^{t\varepsilon}
   e^{B(t-\frac{s}{\varepsilon})}\,dW_s =\int_{-\infty}^t
   e^{B(t-s)}\,d(\frac{1}{\sqrt{\varepsilon}}W_{s\varepsilon}(\omega)) =
   \eta(\theta_t\psi_\varepsilon\omega),
\end{eqnarray}
and
\begin{eqnarray}\label{linear-solu}
\eta^\varepsilon(\omega) =
   \frac{1}{\sqrt{\varepsilon}}\int_{-\infty}^{0}
   e^{-B\frac{s}{\varepsilon}}\,dW_s =\int_{-\infty}^0
   e^{-Bs}\,d(\frac{1}{\sqrt{\varepsilon}}W_{s\varepsilon}(\omega)) =
   \eta(\psi_\varepsilon\omega).
\end{eqnarray}
Moreover, by defining $s=u-t$ at the second equal sign below, we get
\begin{eqnarray}\label{linear-t-solu}
\eta^\eps(\theta_t\omega)=\frac1{\sqrt{\eps}}\int_{-\infty}^t
e^{\frac{B (t-u)}{\eps}}\,dW_u=\frac1{\sqrt{\eps}}\int_{-\infty}^0
e^{-\frac{B s}{\eps}}\,dW_s=\eta^\eps(\omega).
\end{eqnarray}
The equations \eqref{linear-tvar-solu} and \eqref{linear-solu}
indicate that $\eta^\varepsilon(\theta_{t\varepsilon}\omega)$ and
$\eta^\varepsilon(\omega)$ are identically distributed with
$\eta(\theta_t\psi_\varepsilon\omega)$ and
$\eta(\psi_\varepsilon\omega)$, respectively. And by
\eqref{linear-t-solu} and \eqref{linear-solu},
$\eta^\eps(\theta_t\omega)$ and $\eta(\psi_\varepsilon\omega)$ have
the same distribution.

We then introduce a random transformation
\begin{eqnarray}\label{random-transformation}
\begin{matrix}\begin{pmatrix}  X \\ Y \\\end{pmatrix}
:=\mathcal V_\varepsilon(\omega,x, y)=
\begin{pmatrix}  x \\y-\sigma\eta^\varepsilon(\omega)
\\ \end{pmatrix} \end{matrix},
\end{eqnarray}
where $(x, y)$ satisfies  system \eqref{SDE}.\\
Then the SDE system \eqref{SDE} is transferred into the following RDE
system,
\begin{eqnarray}\label{RDE}
\begin{cases}
\dot{X} = AX + f(X, Y+\sigma\eta^\varepsilon(\theta_t\omega)), \,\,\, X\in \mathbb{R}^n,\\
\dot{Y} = \frac{B}{\varepsilon}Y + \frac{1}{\varepsilon}g(X,
Y+\sigma\eta^\varepsilon(\theta_t\omega)),\,\,  Y \in \mathbb{R}^m.
\end{cases}
\end{eqnarray}
By   the variation of constants formula, this RDE system is further
rewritten as
\begin{eqnarray}
X(t)&=&e^{tA}X(0)+\int_0^te^{A(t-s)}f(X(s),
Y(s)+\sigma\eta^\varepsilon(\theta_s\omega))\,ds,\\
Y(t)&=&e^{B\frac{t-t'}{\varepsilon}}Y(t')+\frac{1}{\varepsilon}\int_{t'}^te^{B\frac{t-s}{\varepsilon}}g(X(s),
Y(s)+\sigma\eta^\varepsilon(\theta_s\omega))\,ds.
\end{eqnarray}
As $Y(\cdot)\in C_\lambda^2$, we have the following estimation,
$$| e^{B\frac{t-t'}{\varepsilon}}Y(t')|_{\mathbb{R}^m}\leq e^{\beta\frac{t'-t}{\varepsilon}}|Y(t')|_{\mathbb{R}^m}
=e^{\lambda
t'}|Y(t')|_{\mathbb{R}^m}e^{\frac{t'(\beta-\lambda\varepsilon)-t\beta}{\varepsilon}}\rightarrow
0,\quad t'\rightarrow -\infty,$$  for $\lambda$  satisfying
$\beta-\lambda\varepsilon>0$.\\ Letting $t'\rightarrow -\infty$, we
get the expression of the RDE system \eqref{RDE},
\begin{eqnarray}\label{RDE-solu-x}
X(t)&=&e^{tA}X(0)+\int_0^te^{A(t-s)}f(X(s),
Y(s)+\sigma\eta^\varepsilon(\theta_s\omega))\,ds,\\\label{RDE-solu-y}
Y(t)&=&\frac{1}{\varepsilon}\int_{-\infty}^te^{B\frac{t-s}{\varepsilon}}g(X(s),
Y(s)+\sigma\eta^\varepsilon(\theta_s\omega))\,ds.
\end{eqnarray}
 We rescale the time by letting $\tau=t/\varepsilon$, from
 system \eqref{RDE} and by \eqref{linear-tvar-solu} we get,
\begin{eqnarray}\label{RDE-var-x}
X'(\tau\varepsilon) &=& \varepsilon\big[AX(\tau\varepsilon) +
f(X(\tau\varepsilon),
Y(\tau\varepsilon)+\sigma\eta(\theta_\tau\psi_\varepsilon\omega))\big],
\,\,\,X\in \mathbb{R}^n,\\\label{RDE-var-y} Y'(\tau\varepsilon) &=&
BY + g(X(\tau\varepsilon),
Y(\tau\varepsilon)+\sigma\eta(\theta_\tau\psi_\varepsilon\omega)),\,\,
Y\in \mathbb{R}^m,
\end{eqnarray}
where $'=\frac{d}{d\tau}$.\\
We can rewrite these as the integral form below,
\begin{eqnarray}\label{RDE-var-solu-x}
X(\tau\varepsilon) &=&
X(0)+\varepsilon\int_0^\tau\big[AX(s\varepsilon)+f(X(s\varepsilon),
Y(s\varepsilon)+\sigma\eta(\theta_s\psi_\varepsilon\omega))\big]\,ds,\\
Y(\tau\varepsilon) &=& \int_{-\infty}^\tau e^{B(\tau-s)}
g(X(s\varepsilon),
Y(s\varepsilon)+\sigma\eta(\theta_s\psi_\varepsilon\omega))\,ds.
\end{eqnarray}

\subsection{Dimension reduction via a random slow manifold  }

We now recall some basic facts about random slow manifolds and dimension-reduced systems, when the scale separation is sufficiently large.

A random set $\mathcal{M}(\omega)=\{(x, h(x,
\omega))|x\in \mathbb{R}^n\}$ is
called a random slow manifold (a special random inertial manifold) for the system \eqref{SDE}, if it satisfies the following conditions \cite{Schm}:\\
$(i)$ $\mathcal{M}$ is invariant with respect to a random dynamical system $\phi$, i.e.
$$
\phi(t, \omega, \mathcal{M}(\omega))\subset
\mathcal{M}(\theta_t\omega)\quad for\quad t\geq 0,\quad \omega\in
\Omega.
$$
$(ii)$  $h(x, \omega)$ is globally Lipschitz in $x$ for all
$\omega\in\Omega$ and for any $x\in\mathbb{R}^{n}$ the mapping
$\omega\rightarrow~ h(x, \omega)$ is a random variable.\\
$(iii)$ The   distance of $\phi(t, \omega, z)$ and
$\mathcal{M}(\theta_t\omega)$ tends to $0$ with exponential rate,
for $z\in\mathbb{R}^{n+m}$, as $t$ tends to infinite.


A random slow manifold $\mathcal{M}$, which is lower dimensional, retains the long time dynamics of the original system \eqref{SDE}, when $\varepsilon$ is sufficiently small \cite{FuLiuDuan}.

 In \cite{Schm}, a random Hadamard graph transform was used to
prove the existence of a random inertial manifolds, here we use
Lyapunov- Perron method to achieve our result as in
\cite{FuLiuDuan}.

\begin{lemma}[]  Assume that  $\mathbf{H1}$ and $\mathbf{H2}$ hold  and   that there exists a $\lambda$ such that $\beta-\lambda\varepsilon>0$.
Then, for sufficiently small $\varepsilon$,  there
exists a random slow manifold $\mathcal{\tilde
M}^\varepsilon(\omega)=(\xi,
\tilde h^\varepsilon(\xi, \omega))$ for the random slow-fast system \eqref{RDE}.
\end{lemma}


 \begin{proof}
 This proof is adapted from \cite{FuLiuDuan} for our finite dimensional setting. For completeness, we include the essential part here.
 For a $\lambda>0$, we use the Banach Space $C_\lambda $ as defined in the beginning of this section.\\
Denote a nonlinear mapping
\begin{eqnarray*} \mathcal{T}\big((X, Y),X(0), \omega) =
\bigg\{ e^{tA}X(0)+\int_0^te^{A(t-s)}f(X(s),
Y(s)+\sigma\eta^\varepsilon(\theta_s\omega))\,ds,\\\frac{1}{\varepsilon}\int_{-\infty}^te^{B\frac{t-s}{\varepsilon}}g(X(s),
Y(s)+\sigma\eta^\varepsilon(\theta_s\omega))\,ds \bigg\}.
\end{eqnarray*}

Note that $\mathcal{T}$ is well-defined from
$\mathbb{R}^{n+m}\times
\mathbb{R}^n\times\Omega\rightarrow\mathbb{R}^{n+m}$. We will show
that for every   initial data $X(0)=\xi\in \mathbb{R}^n$,
\eqref{RDE} have a unique solution in $C_\lambda $. For   $(X,
Y)$, $(\bar{X},
\bar{Y})\in C_\lambda $, we have that\\
\begin{eqnarray*}
\lefteqn{\Big|\mathcal{T}\big((X, Y), \xi, \omega) -
\mathcal{T}\big((\bar{X}, \bar{Y}), \xi, \omega)\Big|_{C_\lambda }}\\
&=&\sup\limits_{t\leq 0}e^{\lambda
t}\Big|\int_0^te^{A(t-s)}\big[f(X(s),
Y(s)+\sigma\eta^\varepsilon(\theta_s\omega))-f(\bar{X}(s),
\bar{Y}(s)+\sigma\eta^\varepsilon(\theta_s\omega))\big]\,ds\Big|_{\mathbb{R}^n}\\{}&&+\sup\limits_{t\leq
0}e^{\lambda
t}\Big|\frac{1}{\varepsilon}\int_{-\infty}^te^{B\frac{t-s}{\varepsilon}}\big[g(X(s),
Y(s)+\sigma\eta^\varepsilon(\theta_s\omega))-g(\bar{X}(s),
\bar{Y}(s)+\sigma\eta^\varepsilon(\theta_s\omega))\big]\,ds\Big|_{\mathbb{R}^m}\\
&\leq& \Big|(X, Y)-(\bar{X}, \bar{Y})\Big|_{C_\lambda }
\sup\limits_{t\leq 0}\Big\{e^{(\lambda+\alpha)
t}KL_f\int_t^0e^{-(\lambda+\alpha) s}\,ds\Big\}\\{}&&+\Big|(X,
Y)-(\bar{X}, \bar{Y})\Big|_{C_\lambda } \sup\limits_{t\leq
0}\Big\{\frac{1}{\varepsilon}e^{(\lambda-\frac{\beta}{\varepsilon})
t}KL_g\int_{-\infty}^te^{(\frac{\beta}{\varepsilon}-\lambda)s}\,ds\Big\}\\
&\leq& \big(\frac{KL_f}{\alpha+\lambda} +
\frac{KL_g}{\beta-\varepsilon\lambda}\big)\Big|(X, Y)-(\bar{X},
\bar{Y})\Big|_{C_\lambda }.
\end{eqnarray*}
The first inequality is   by  $\mathbf{H1}$ and the Lipschitz
continuity of $f$ and $g$, while the second inequality comes from direct
calculation.   Taking $~\lambda=\frac{\beta-KL_g}{2\varepsilon}>0$,
which satisfies $\beta-~\lambda\varepsilon=~\frac{\beta+KL_g}{2}>0$,
we conclude that
\begin{eqnarray*}
\lefteqn{\Big|\mathcal{T}\big((X, Y), \xi, \omega\big) -
\mathcal{T}\big((\bar{X}, \bar{Y}), \xi,
\omega\big)\Big|_{C_\frac{\beta-KL_g}{2\varepsilon} }}\\ &\leq&
\big(\frac{2
KL_f\varepsilon}{2\varepsilon\alpha+\beta-KL_g}+\frac{2KL_g}{KL_g+\beta}\big)\Big|(X,
Y)-(\bar{X}, \bar{Y})\Big|_{C_\frac{\beta-KL_g}{2\varepsilon} }.
\end{eqnarray*}
By the assumption  $\mathbf{H2}$, $\frac{2KL_g}{KL_g+\beta}<1$, and
$\frac{2 KL_f\varepsilon}{2\varepsilon\alpha+\beta-KL_g}\rightarrow
0$ as $\varepsilon\rightarrow 0$. Therefore, for $\varepsilon$ small
enough, $\big(\frac{2K
L_f\varepsilon}{2\varepsilon\alpha+\beta-KL_g}+\frac{2KL_g}{KL_g+\beta}\big)<1$.
The contraction map theorem implies that for every $\xi\in
\mathbb{R}^n$, $\mathcal{T}\big((X, Y), \xi, \omega\big)$ has a
fixed point $(X(t), Y(t))\in C_\frac{\beta-KL_g}{2\varepsilon} $
which is the unique solution of the differential equation system
\eqref{RDE}. Moreover, the fixed point has the property
\begin{eqnarray}\label{x-x}
\nonumber\lefteqn{\Big|(X(\cdot; \xi, \omega), Y(\cdot; \xi,
\omega))-(X(\cdot; \bar\xi, \omega), Y(\cdot; \bar\xi,
\omega))\Big|_{C_\frac{\beta-KL_g}{2\varepsilon}
}}\\&\leq&\frac{K}{1-\big(\frac{2K
L_f\varepsilon}{2\varepsilon\alpha+\beta-KL_g}+\frac{2KL_g}{KL_g+\beta}\big)}\Big|\xi-\bar\xi\Big|_{\mathbb{R}^n}.
\end{eqnarray}
Denoting $\tilde h^\varepsilon(\xi, \omega)=Y(0, \xi, \omega)$,   we obtain
\begin{eqnarray}\label{smf_h^eps}
\tilde{h}^\varepsilon(\xi, \omega) =
\frac{1}{\varepsilon}\int_{-\infty}^0e^{-B\frac{s}{\varepsilon}}g(X(s),
Y(s)+\sigma\eta^\varepsilon(\theta_s\omega))\,ds, \quad \xi \in
\mathbb{R}^n.
\end{eqnarray}
With the help of inequality \eqref{x-x}, we   further have
$$\Big|\tilde{h}^\varepsilon(\xi, \omega)-\tilde{h}^\varepsilon(\bar\xi,
\omega)\Big|_{\mathbb{R}^m}\leq\frac{2K^2L_g}{KL_g+\beta}\cdot
\frac{\big|\xi-\bar\xi\big|_{\mathbb{R}^n}}{1-\big(\frac{2K
L_f\varepsilon}{2\varepsilon\alpha+\beta-KL_g}+\frac{2KL_g}{KL_g+\beta}\big)}.
$$
Thus, $\tilde h^\varepsilon$ is Lipschitz continuous. By the fact that $(X(0),
Y(0))\in \mathcal{\tilde M}^\varepsilon(\omega)$ if and only if
there exists $(X, Y)\in C_\lambda$ and satisfies \eqref{RDE-solu-x}
and \eqref{RDE-solu-y}, it follows that $(X(0), Y(0))\in
\mathcal{\tilde M}^\varepsilon(\omega)$ if and only if there exists
$\xi\in\mathbb{R}^n$ such that $(X(0), Y(0))=(\xi, \tilde
h^\varepsilon(\xi, \omega))$. Therefore, there exists a random slow manifold
$$\mathcal{\tilde M}^\varepsilon(\omega)=\{(\xi, \tilde h^\varepsilon(\xi,
\omega)\mid\xi\in\mathbb{R}^n\}.$$
\end{proof}

By the random transformation \eqref{random-transformation} and noting that    $\eta^\eps(\theta_t\omega)$ and
$\eta(\psi_\varepsilon\omega)$ have the same distribution,   the dynamics on the   slow manifold is now
described by the following dimension-reduced system in $\R^n$ (from
equation (\ref{SDE}) ), for $\varepsilon$ sufficiently small:
\begin{equation} \label{slowdynamics}
\dot \xi = A \xi + f(\xi, \sigma\eta(\psi_\varepsilon\omega)+\tilde
h^\varepsilon(\xi, \theta_t\omega) ), \;\; \xi\in\mathbb{R}^n.
\end{equation}


\subsection{Approximation of a random slow manifold}

We now   approximate the slow manifolds for sufficiently small $\varepsilon$.
Expand   the solution of system \eqref{RDE-var-y} as
\begin{eqnarray}\label{Y(tvar)}
Y(\tau\varepsilon)=Y_0(\tau)+\varepsilon Y_1(\tau)+\varepsilon^2
Y_2(\tau)+\cdots,
\end{eqnarray}
and the initial conditions as
$$Y(0)=\tilde{h}^\varepsilon(\xi, \omega)=\tilde h^{(0)}(\xi, \omega)+\varepsilon \tilde{h}^{(1)}(\xi, \omega)+\cdots, $$
and $X(0)=\xi \in \mathbb{R}^n$.  With the help of
\eqref{RDE-var-solu-x} and \eqref{Y(tvar)}, we have the expansions
\begin{eqnarray*}
f(X(\tau\varepsilon),
Y(\tau\varepsilon)+\sigma\eta(\theta_\tau\psi_\varepsilon\omega))&=&f(\xi,
Y_0(\tau)+\sigma\eta(\theta_\tau\psi_\varepsilon\omega))+f_x(\xi,Y_0(\tau)+\sigma\eta(\theta_\tau\psi_\varepsilon\omega))\\{}&&\cdot\varepsilon\int_0^\tau\big[A
X(s\varepsilon)+f(X(s\varepsilon),
Y(s\varepsilon)+\sigma\eta(\theta_s\psi_\varepsilon\omega))\big]\,ds\\{}&&+f_y(\xi,
Y_0(\tau)+\sigma\eta(\theta_\tau\psi_\varepsilon\omega))\cdot\big[\varepsilon
Y_1(\tau)+\cdots\big]+\cdots\\&=&f(\xi,
Y_0(\tau)+\sigma\eta(\theta_\tau\psi_\varepsilon\omega))\\{}&&+f_x(\xi,
Y_0(\tau)+\sigma\eta(\theta_\tau\psi_\varepsilon\omega))\cdot
\varepsilon\int_0^\tau\big[A \xi+f(\xi,
Y_0(s)+\sigma\eta(\theta_s\psi_\varepsilon\omega))\big]\,ds\\{}&&+f_y(\xi,
Y_0(\tau)+\sigma\eta(\theta_\tau\psi_\varepsilon\omega))\cdot\varepsilon
Y_1(\tau)+\cdots,\\
g(X(\tau\varepsilon),
Y(\tau\varepsilon)+\sigma\eta(\theta_\tau\psi_\varepsilon\omega))&=&g(\xi,
Y_0(\tau)+\sigma\eta(\theta_\tau\psi_\varepsilon\omega))\\{}&&+g_x(\xi,
Y_0(\tau)+\sigma\eta(\theta_\tau\psi_\varepsilon\omega))\cdot
\varepsilon\int_0^\tau\big[A \xi+f(\xi,
Y_0(s)+\sigma\eta(\theta_s\psi_\varepsilon\omega))\big]\,ds\\{}&&+g_y(\xi,
Y_0(\tau)+\sigma\eta(\theta_\tau\psi_\varepsilon\omega))\cdot\varepsilon
Y_1(\tau)+\cdots.
\end{eqnarray*}
Inserting \eqref{Y(tvar)} into \eqref{RDE-var-y},  expanding
\eqref{RDE-var-y} and then matching the terms of the same power of
$\varepsilon$, we get
\begin{eqnarray}\label{Y_0}
\begin{cases}
Y'_0(\tau)=BY_0(\tau)+g(\xi,
Y_0(\tau)+\sigma\eta(\theta_\tau\psi_\varepsilon\omega)), \\
Y_0(0)=\tilde h^{(0)}(\xi, \omega),
\end{cases}
\end{eqnarray}
and
\begin{eqnarray}\label{Y_1}
\begin{cases}
Y'_1(\tau)=\big[B+g_y(\xi,
Y_0(\tau)+\sigma\eta(\theta_\tau\psi_\varepsilon\omega))\big]Y_1(\tau)\\
\quad\quad\quad\quad+g_x(\xi,
Y_0(\tau)+\sigma\eta(\theta_\tau\psi_\varepsilon\omega))\big\{A\tau\xi+\int_0^\tau
f(\xi, Y_0(s)+\sigma\eta(\theta_s\psi_\varepsilon\omega))\,ds\big\} , \\
Y_1(0)=\tilde h^{(1)}(\xi, \omega).
\end{cases}
\end{eqnarray}
Solving the two equations for $Y_0(\tau)$ and $Y_1(\tau)$, we obtain
\begin{eqnarray}\label{Y_0(t)}
Y_0(\tau)=e^{B\tau}\tilde h^{(0)}(\xi, \omega)+\int_0^\tau
e^{-B(s-\tau)}g(\xi,
Y_0(s)+\sigma\eta(\theta_s\psi_\varepsilon\omega))\,ds,
\end{eqnarray}
and
\begin{eqnarray}\nonumber
Y_1(\tau)&=&e^{B\tau+\int_0^\tau g_y(\xi,
Y_0(s)+\sigma\eta(\theta_s\psi_\varepsilon\omega))\,ds}\tilde
h^{(1)}(\xi, \omega)+\int_0^\tau e^{-B(s-\tau)+\int_s^\tau g_y(\xi,
Y_0(r)+\sigma\eta(\theta_r\psi_\varepsilon\omega))\,dr}\\{}&&\cdot
g_x(\xi,
Y_0(s)+\sigma\eta(\theta_\tau\psi_\varepsilon\omega))\big[As\xi+\int_0^s\big(f(\xi,
Y_0(r)+\sigma\eta(\theta_r\psi_\varepsilon\omega))\,dr\big]\,ds.\nonumber\\
\end{eqnarray}
With the help of \eqref{RDE-var-solu-x} and \eqref{Y(tvar)}, the
expression \eqref{smf_h^eps} can be calculated as follows
\begin{eqnarray*}
\tilde{h}^\varepsilon(\xi, \omega) &=&
\frac{1}{\varepsilon}\int_{-\infty}^0e^{-B\frac{s}{\varepsilon}}g(X(s),
Y(s)+\sigma\eta^\varepsilon(\theta_s\omega))\,ds \\
&=&\int_{-\infty}^0 e^{-Bs}g(X(s\varepsilon),
Y(s\varepsilon)+\sigma\eta(\theta_s\psi_\varepsilon\omega))\,ds \\
&=& \int_{-\infty}^0 e^{-Bs}\Big\{g(\xi,
Y_0(s)+\sigma\eta(\theta_s\psi_\varepsilon\omega))+g_x(\xi,
Y_0(s)+\sigma\eta(\theta_s\psi_\varepsilon\omega))\varepsilon\big[As\xi\\{}&&
+\int_0^s\big(f(\xi,
Y_0(r)+\sigma\eta(\theta_r\psi_\varepsilon\omega))\,dr\big]
+g_y(\xi, Y_0(s)+\sigma\eta(\theta_s\psi_\varepsilon\omega))\varepsilon Y_1(s)\Big\}\,ds+\mathcal{O}(\varepsilon^2)\\
&=& \int_{-\infty}^0 e^{-Bs}g(\xi,
Y_0(s)+\sigma\eta(\theta_s\psi_\varepsilon\omega))\,ds
\\{}&&+ \varepsilon\int_{-\infty}^0 e^{-Bs}\Big\{g_x(\xi,
Y_0(s)+\sigma\eta(\theta_s\psi_\varepsilon\omega))\big[As\xi+\int_0^s\big(f(\xi,
Y_0(r)+\sigma\eta(\theta_r\psi_\varepsilon\omega))\big)\,dr\big]\\{}&&\quad\quad\quad\quad\quad\quad
+g_y(\xi,
Y_0(s)+\sigma\eta(\theta_s\psi_\varepsilon\omega))Y_1(s)\Big\}\,ds+\mathcal{O}(\varepsilon^2).
\end{eqnarray*}
To get the second equation, we used $\tau=s/\varepsilon$ and then
used $s$ to replace $\tau$. Thus the zero and first order terms in
$\varepsilon$, of $\tilde h^\varepsilon$ in the random slow manifold
 $\mathcal{M}^\varepsilon(\omega)$ for \eqref{RDE},  are
respectively
\begin{eqnarray}\label{t-h^d}
\tilde h^{(0)}(\xi, \omega)=\int_{-\infty}^0 e^{-Bs}g(\xi,
Y_0(s)+\sigma\eta(\theta_s\psi_\varepsilon\omega))\,ds,
\end{eqnarray}
and
\begin{eqnarray}\label{t-h^(1)}
\tilde h^{(1)}(\xi, \omega)&=&\nonumber\int_{-\infty}^0
e^{-Bs}\Big\{g_x(\xi,
Y_0(s)+\sigma\eta(\theta_s\psi_\varepsilon\omega))\big[As\xi+\int_0^sf(\xi,
Y_0(r)+\sigma\eta(\theta_r\psi_\varepsilon\omega))\,dr\big]\\{}&&\quad\quad\quad\quad\quad\quad
+g_y(\xi,
Y_0(s)+\sigma\eta(\theta_s\psi_\varepsilon\omega))Y_1(s)\Big\}\,ds.
\end{eqnarray}
That is, the slow manifolds $\mathcal{\tilde
M}^\varepsilon(\omega)=\{(\xi, \tilde h^\varepsilon(\xi, \omega))\}$
of \eqref{RDE} up to the order $\mathcal {O}(\varepsilon^2)$ is
represented by
 $\tilde h^\varepsilon(\xi,
\omega)=\tilde h^{(0)}(\xi, \omega)+\varepsilon\tilde h^{(1)}(\xi,
\omega)+\mathcal{O}(\varepsilon^2)$. This produces an approximation
of the random slow manifold.

 Therefore, we have the following result.

\begin{theorem}[Approximation of a random slow manifold] \label{slow999}
 Assume that  $\mathbf{H1}$ and $\mathbf{H2}$ hold  and assume that there is a $\lambda$ such that $\beta-\lambda\varepsilon>0$.
 Then, for sufficiently small $\varepsilon$,  there
exists a  slow manifold  $\mathcal{\tilde M}^\varepsilon(\omega)
=\{(\xi, \tilde h^\varepsilon(\xi,\omega)\mid\xi\in\mathbb{R}^n\},$
where $\tilde h^\varepsilon(\xi,\omega)=~ \tilde h^{(0)}(\xi,
\omega)+~\varepsilon\tilde h^{(1)}(\xi, \omega)
+~\mathcal{O}(\varepsilon^2)$ with $\tilde h^{(0)}(\xi, \omega)$ and
$\tilde h^{(1)}(\xi, \omega)$  expressed in \eqref{t-h^d} and
\eqref{t-h^(1)}, respectively.
\end{theorem}
With the approximated random slow manifold
\begin{equation}
 \hat{h}^\varepsilon(\xi,\omega)=~ \tilde h^{(0)}(\xi,
\omega)+~\varepsilon\tilde h^{(1)}(\xi, \omega),
\end{equation}
we obtain
the following dimension-reduced approximate random system in $\R^n$ (from equation (\ref{slowdynamics}) ), for $\varepsilon$ sufficiently small:
\begin{equation} \label{slowdynamics2}
\dot \xi = A \xi + f(\xi, \sigma\eta(\psi_\varepsilon\omega)+
\hat{h}^\varepsilon(\xi, \theta_t\omega) ), \;\; \xi\in\mathbb{R}^n.
\end{equation}

\section{Settling of inertial particles under random influences}
\label{settle888}

For the motion of aerosol particles in a cellular flow field,
Stommel once observed that, ignoring particle inertial ($\eps=0$),
some particles follow closed paths and are permanently suspended in
the flow. Rubin, Jones and Maxey
 \cite{Rubin} showed that   any small amount inertial (small $\eps>0$) will cause almost all particles to settle.
Via a singular perturbation theory \cite{Jones}, Jones showed the
existence of an attracting slow manifold.  By analyzing the
equations of motion on the slow manifold, especially heteroclinic
orbits, they established the presence of mechanisms that inhibit
trapping and enhance settling of particles.

Let us now  examine  the motion of aerosol particles in a
 random cellular flow field, using the slow manifold reduction technique developed in the previous section.


Consider a model for the motion of  aerosol particles  in a cellular flow field, under random environmental influences \cite{Rubin}
\begin{equation}\label{iner-sto-slow}
\begin{cases}
    \dot{y}_1 =  v_1, \\
    \dot{y}_2 =  v_2, \\
    \dot{v}_1 =  -\frac{1}{\varepsilon} v_1 + \frac{1}{\varepsilon} a \sin y_1 \cos y_2 +  \frac{\sigma}{\sqrt{\eps}} \dot{W}_t^1,\\
    \dot{v}_2 = -\frac{1}{\varepsilon} v_2 + \frac{1}{\varepsilon}(V- a \cos y_1\sin y_2) +  \frac{\sigma}{\sqrt{\eps}} \dot{W}_t^2,
\end{cases}
\end{equation}
where $(y_1, y_2)$ and $(v_1, v_2)$ are position  and velocity,
respectively, of a particle in the horizontal- vertical plane
(positive $y_2$ axis points   to the  settling/gravitational
direction), $a$ is a velocity scale, and $V$ is the settling
velocity in still fluid. Moreover, $W_t^1, W_t^2$ are independent
scalar Wiener processes,
  $\sigma$ is a  positive parameter, and  $\varepsilon$ is the inertial response time scale of the particle.
  Note that $(a \sin y_1 \cos y_2, -a \cos y_1\sin y_2)$ is the so-called cellular flow
  field velocity components (horizontal and vertical)
  on the domain (a `cell') $D\triangleq (0, \pi)\times (0, \pi)$.

  As in Section \ref{slow888}, this four dimensional SDE system can be converted to the following   RDE system
\begin{eqnarray}\label{iner-RDE}
\begin{cases}
    \dot{y}_1 =  v_1 + \sigma\eta_1^\varepsilon(\theta_t\omega), \\
    \dot{y}_2 =  v_2 + \sigma\eta_2^\varepsilon(\theta_t\omega), \\
    \dot{v}_1 = -\frac{1}{\varepsilon} v_1 + \frac{1}{\varepsilon} a \sin y_1 \cos y_2 ,\\
    \dot{v}_2 = -\frac{1}{\varepsilon} v_2 + \frac{1}{\varepsilon}(V- a \cos y_1 \sin y_2),
\end{cases}
\end{eqnarray}
where
$$\eta_1^\varepsilon(\theta_t\omega)=\frac{1}{\sqrt{\varepsilon}}\int_{-\infty}^t
   e^{\frac{-1}{\varepsilon}(t-s)}\,dW_s^1, \quad\quad
   \eta_2^\varepsilon(\theta_t\omega)=\frac{1}{\sqrt{\varepsilon}}\int_{-\infty}^t
   e^{\frac{-1}{\varepsilon}(t-s)}\,dW_s^2.$$
 Denoting $y_1(0)=\xi_1$ and $y_2(0)=\xi_2$, and we examine the motion of  the particle  $(\xi_1, \xi_2)$. By using
\eqref{t-h^d} and \eqref{Y_0(t)}, we get
$$\tilde h^{(0)}_1(\xi, \omega)=\int_{-\infty}^0 e^s a \sin \xi_1 \cos \xi_2\,ds
=a \sin \xi_1 \cos \xi_2, \quad \tilde h^{(0)}_2(\xi, \omega)=V-a
\cos \xi_1 \sin \xi_2,$$ and $$v_{10}(t)=a \sin \xi_1 \cos \xi_2,
\quad v_{20}(t)=V-a \cos \xi_1 \sin \xi_2.$$ Owing to
\eqref{t-h^(1)},
\begin{eqnarray*}
\tilde h_1^{(1)}(\xi_1,\xi_2,\omega) &=&\int_{-\infty}^0 e^s\big[a^2
s\sin \xi_1 \cos \xi_1 -a V s \sin \xi_1 \sin \xi_2  +a  \cos \xi_1
\cos \xi_2\,
 \sigma\int_0^s\eta_1(\theta_r\psi_\varepsilon\omega)\,dr \\
&&{}\quad -a \sin \xi_1 \sin \xi_2\,
  \sigma\int_0^s\eta_2(\theta_r\psi_\varepsilon\omega)\,dr\big]\,ds\\{}&=&
-a^2 \sin \xi_1 \cos \xi_1 +a V  \sin \xi_1 \sin \xi_2 + a \sigma
\cos \xi_1 \cos \xi_2  \int_{-\infty}^0 s e^s \,dW_s^1 \\&&{} -a
\sigma \sin \xi_1
\sin \xi_2  \int_{-\infty}^0 s e^s \,dW_s^2,\\
\tilde h_2^{(1)}(\xi_1,\xi_2,\omega) &=& -a^2 \sin \xi_2 \cos \xi_2
+a V \cos \xi_1 \cos \xi_2 + a \sigma \sin \xi_1 \sin \xi_2
\int_{-\infty}^0 s e^s \,dW_s^1 \\&&{} -a \sigma\cos \xi_1 \cos
\xi_2 \int_{-\infty}^0 s e^s \,dW_s^2.
\end{eqnarray*}
Therefore, from (\ref{slowdynamics}),  the dynamics on the random slow manifold  is described by the following dimension-reduced system:
\begin{eqnarray}\label{reduced1}
\dot\xi_1 \nonumber&=& \tilde h_1(\xi_1,\xi_2,\omega)
+\sigma\eta_1(\psi_\varepsilon\omega)
\\{}&=& \sigma\int_{-\infty}^0
   e^{s}\,dW_s^1(\psi_\varepsilon \omega)+a
\sin\xi_1 \cos\xi_2 +
  \varepsilon\big\{-a^2 \sin\xi_1 \cos\xi_1 +a V \sin\xi_1 \sin\xi_2 \nonumber\\&&+ a \sigma
\cos\xi_1 \cos\xi_2  \int_{-\infty}^0 s e^s \,dW_s^1 -a
\sigma\sin\xi_1 \sin\xi_2  \int_{-\infty}^0 s e^s \,dW_s^2\big\},
\end{eqnarray}
\begin{eqnarray}\label{reduced2}
\dot\xi_2 \nonumber &=& \tilde
h_2(\xi_1,\xi_2,\omega)+\sigma\eta_2(\psi_\varepsilon\omega)\\{}&=&\sigma\int_{-\infty}^0
   e^{s}\,dW_s^2(\psi_\varepsilon \omega)+ V-a
\cos\xi_1 \sin\xi_2 + \varepsilon \big\{ -a^2 \sin\xi_2 \cos\xi_2 +a
V \cos\xi_1 \cos\xi_2 \nonumber\\&&+ a \sigma \sin\xi_1 \sin\xi_2
 \int_{-\infty}^0 s e^s \,dW_s^1
 -a \sigma\cos\xi_1 \cos\xi_2  \int_{-\infty}^0 s e^s \,dW_s^2
\big\}.
\end{eqnarray}

\bigskip

\subsection{Numerical simulation: First exit time and escape probability}

Note that $(\xi_1, \xi_2)$ is the particle position. For random slow
manifold reduction, it is customary to use a notation different from
the original one $(y_1, y_2)$. The positive $\xi_2$ direction points
toward the bottom of the fluid.

In this section, we conduct numerical simulations for this reduced
or slow system \eqref{reduced1}-\eqref{reduced2}. When
$\varepsilon=0$, $\sigma=0$, this reduced system  becomes the
classical system for the motion of particles in the cellular flow.
When $\varepsilon \neq 0$, $\sigma=0$ indicates no noise, while a
non-zero $\sigma$ means   noise is present.

\paragraph{Simulation}
Motivated by understanding the settling of particles as in
\cite{Rubin}, we first calculate first exit time of particles,
described by the random system \eqref{reduced1}-\eqref{reduced2},
from the domain $D \triangleq (0, \pi)\times (0, \pi)$ and then
examine how particles, exit or escape the fluid domain $D$. To this
end, we introduce two concepts: First exit time and escape
probability. The first exit time is the time when a particle,
initially at $(\xi_1, \xi_2)\in D$,  first exits the domain $D$:
\begin{equation*}
\tau(\xi_1, \xi_2) \triangleq \inf\{t: \; (\xi_1(t), \xi_2(t))
\notin D \}.
\end{equation*}
Let $\Gamma \subset \partial D$ be a subboundary.  The escape
probability  $P_{\Gamma}(\xi_1, \xi_2)$, for a particle initially at
$(\xi_1, \xi_2)$,  through a subboundary  $\Gamma$, is the
likelihood that this particle first escapes the domain $D$ by
passing through $\Gamma$. We will take $\Gamma$ to be one of the
four sides of the fluid domain $D$. 
The escape probability of a particle through the top side
$\Gamma=\{\xi_2=\pi\}$ means the likelihood that this particle
settles directly to the bottom of the fluid (note that the positive
$\xi_2$ direction points to the bottom of the fluid).

To compute the first exit time from the domain $D$,   we place
particles on a lattice of grid points in  $D$ and on its boundary,
and set a large enough threshold time $T$. As soon as a particle
reaches boundary of $D$, it is regarded as `having exited' from $D$.
If a particle leaves $D$ before $T$, then the time of leaving is
taken as the first exit time, but if it is still in the domain at
time $T$, we take $T$ as the first exit time. When a particle's
first exit time is $T$, we can see it as trapped in the cell.

In order to calculate the escape probability of a particle under
noise through a subboundary $\Gamma$, one of the four sides of the
domain, we calculate a large number, $N$, of paths for each particle
to see how many (say $M$) of them exit through the subboundary
$\Gamma$, and then we get the escape probability $\frac{M}{N}$. We
do this for particles placed on a lattice of grid points in $D$ and
on its boundary. When a particle reaches or is on a side
subboundary, it is regarded as `having escaped through' that part of
the boundary.

\begin{figure}\center
\includegraphics[height=6cm]{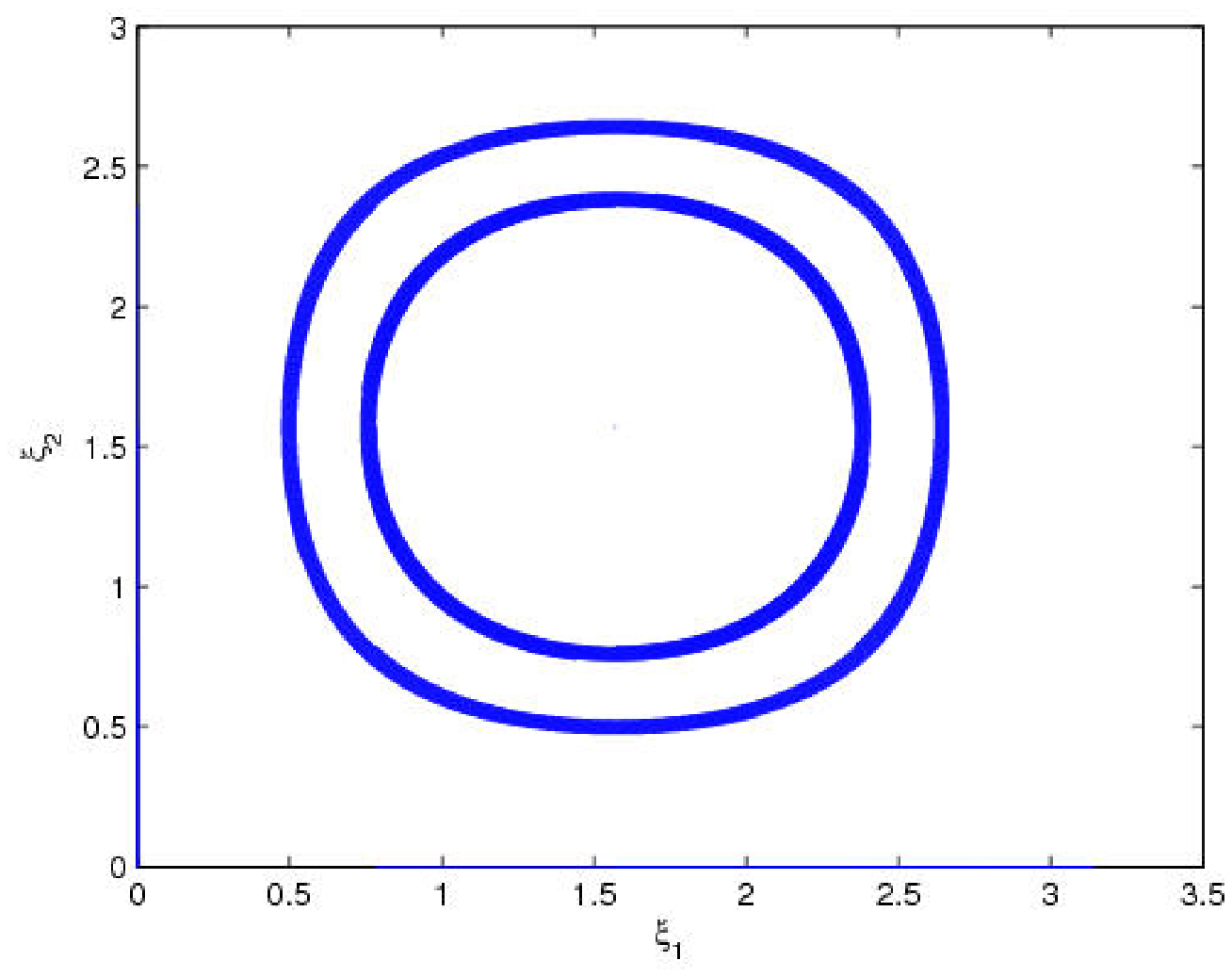}\\
\includegraphics[height=6cm]{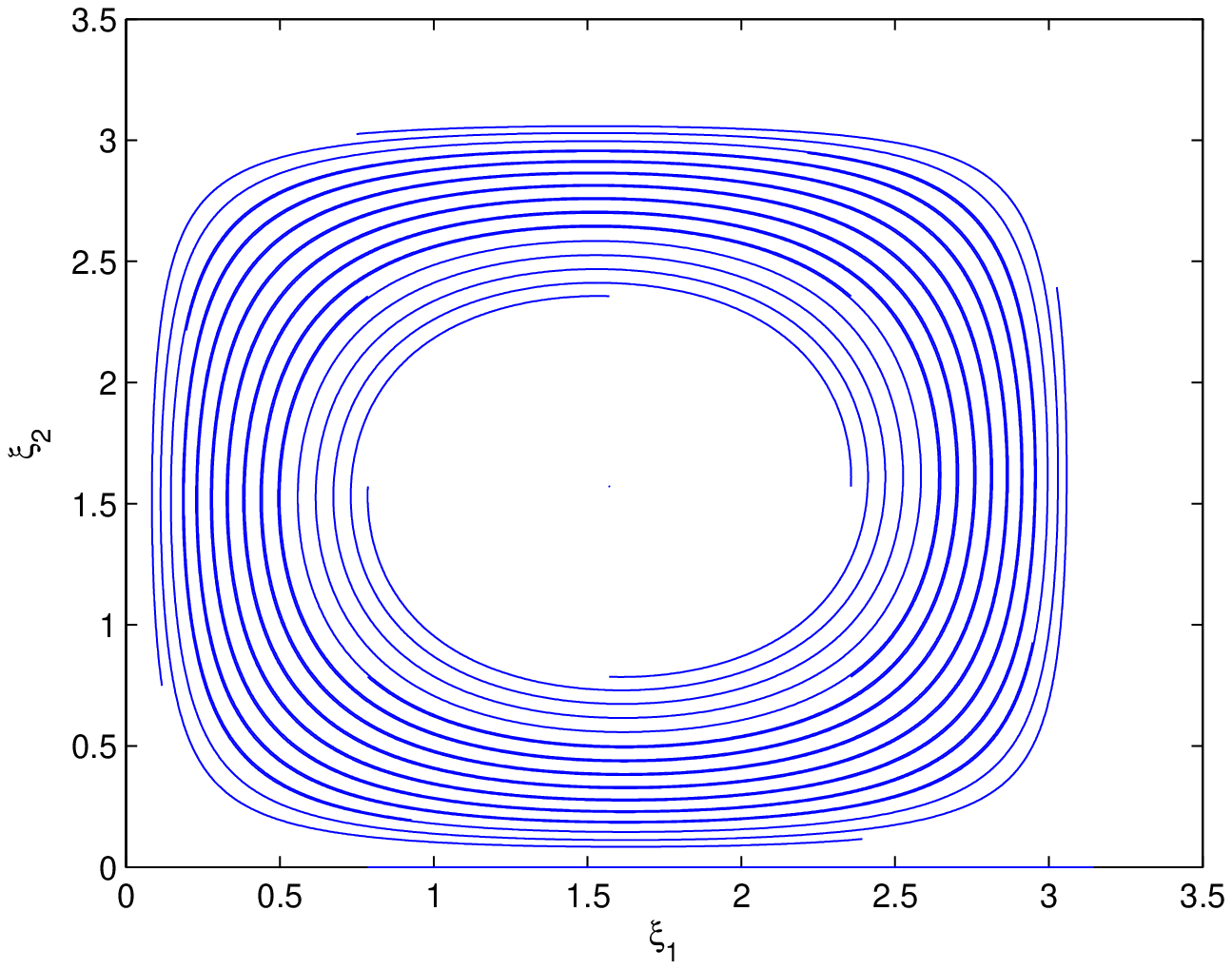}
\caption{Phase portraits of the slow system
\eqref{reduced1}-\eqref{reduced2} with $V=0$,   $a=0.7$ and $\sigma=0$ (no noise):
$\varepsilon=0$  (top); $\varepsilon=0.05$  (bottom).} \label{w0}
\end{figure}

\begin{figure}\center
\includegraphics[height=6cm]{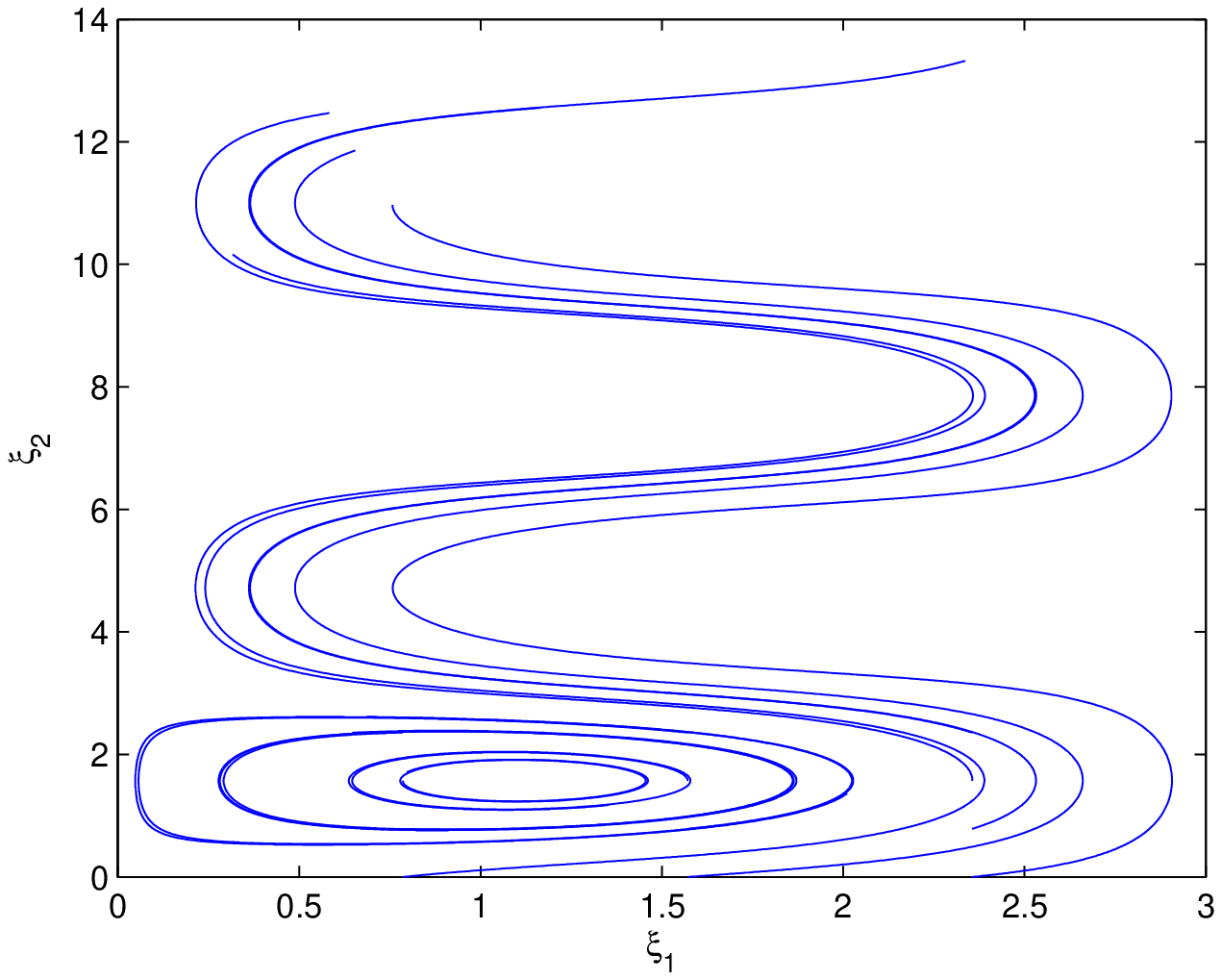}\\
\includegraphics[height=6cm]{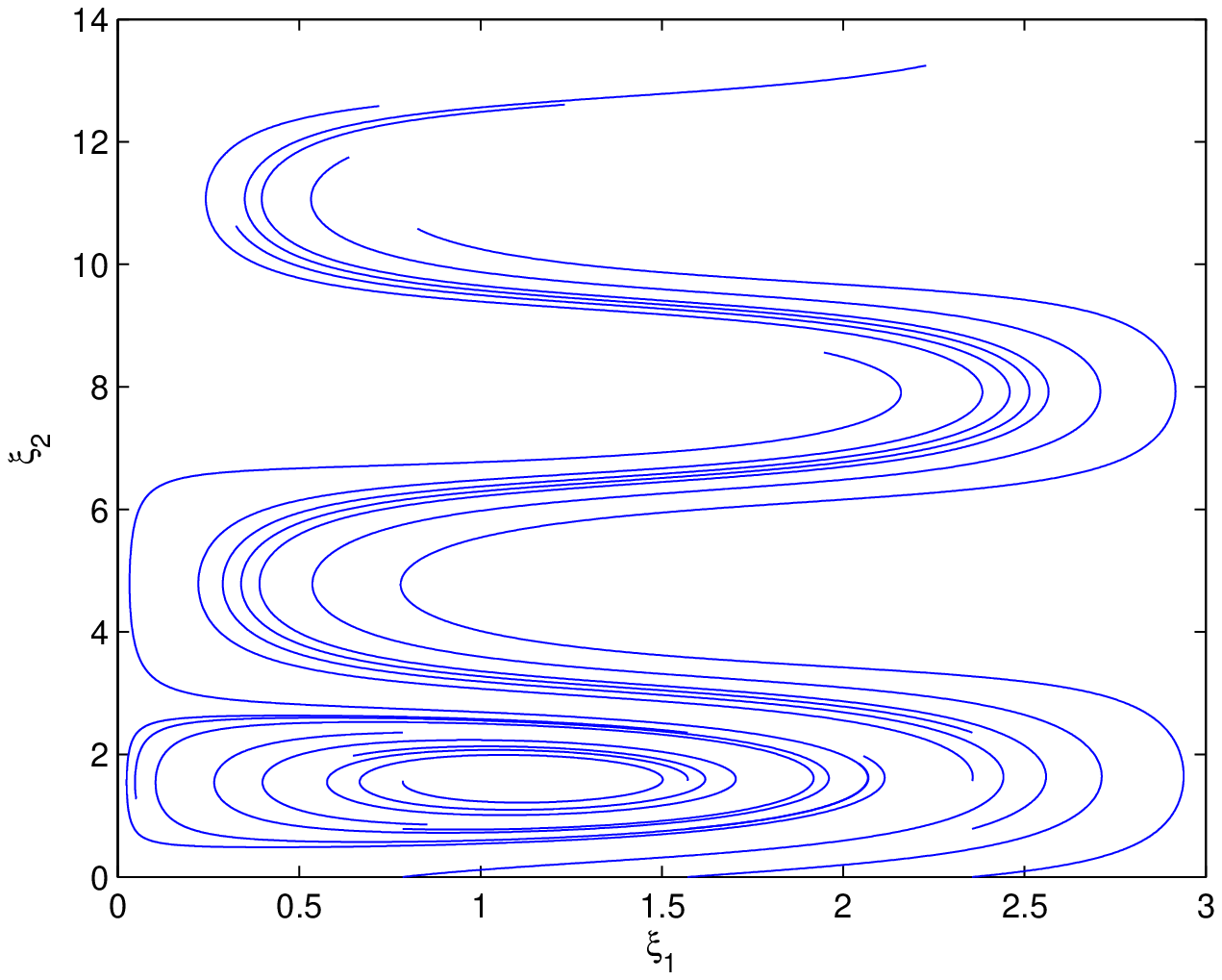}\\
\includegraphics[height=6cm]{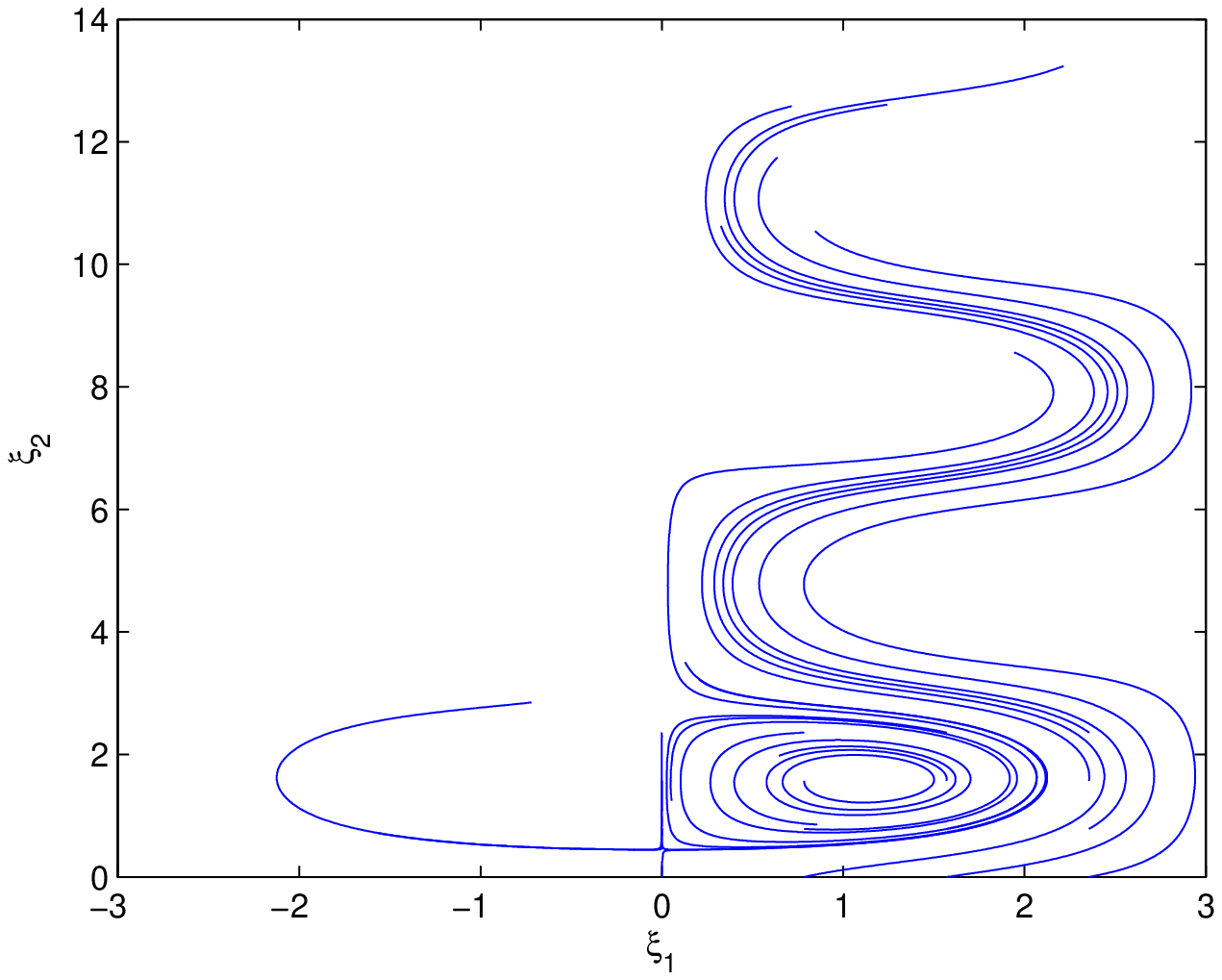}
\caption{Sample particle orbits of the slow system  \eqref{reduced1} and
\eqref{reduced2} with $V=0.3$ and $a=0.7$: $\varepsilon=0$ and
$\sigma=0$ (top, no noise); $\varepsilon=0.05$ and $\sigma=0$
(middle, no noise); $\varepsilon=0.05$ and $\sigma=0.01$ (bottom,
with noise).} \label{w03}
\end{figure}

\begin{figure}
\subfigure(a)\includegraphics[height=6cm]{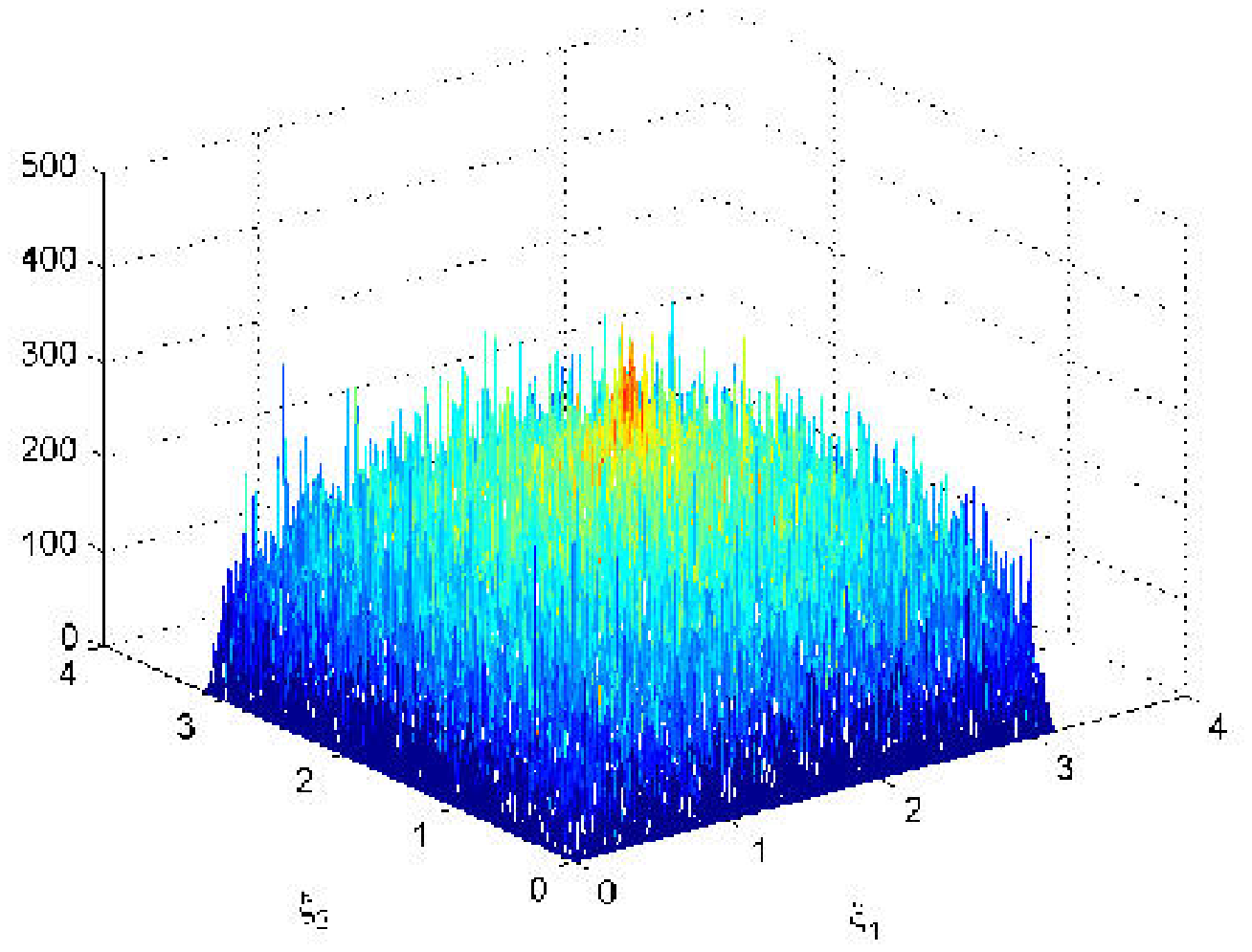}
\subfigure(b)\includegraphics[height=6cm]{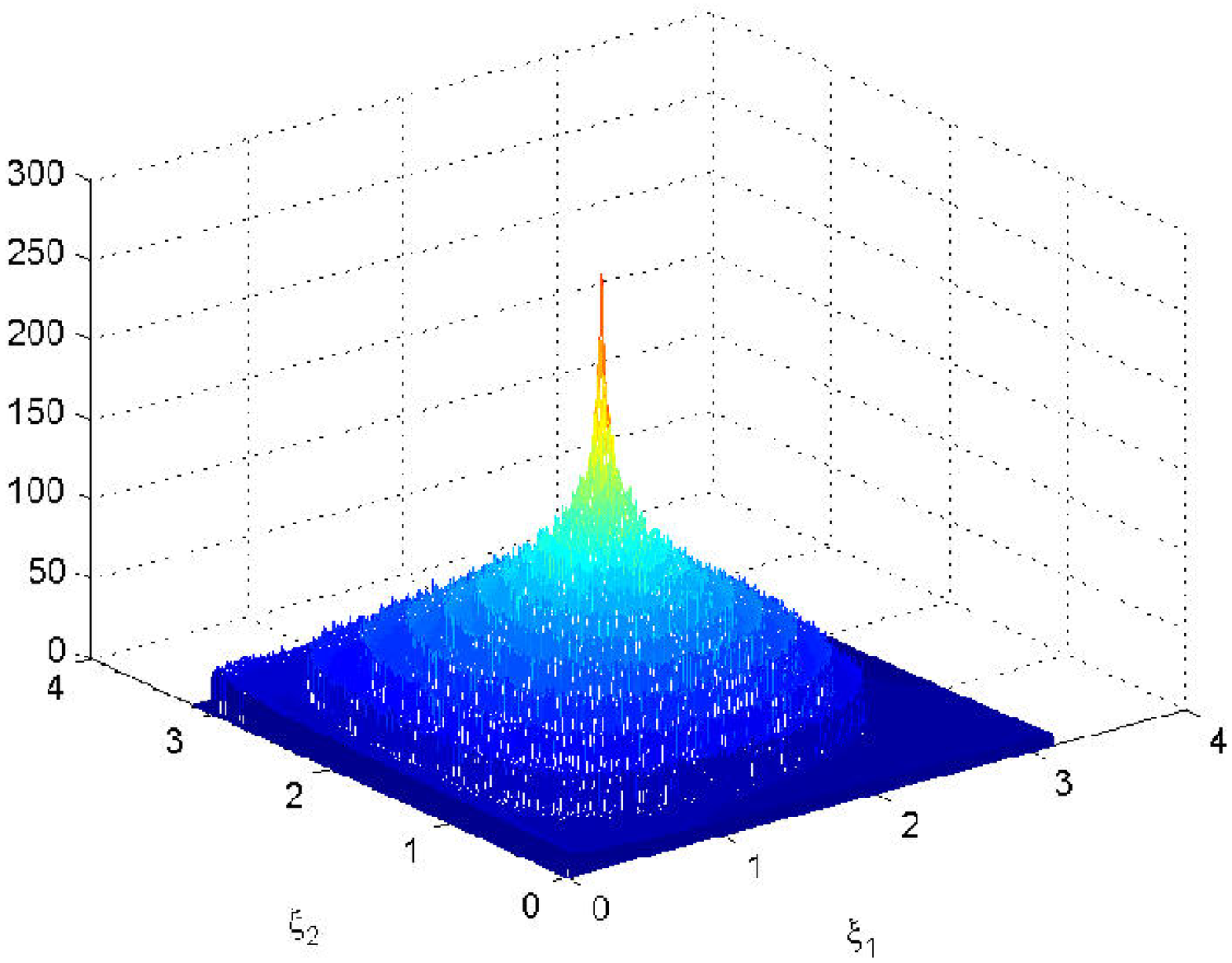}\\
\subfigure(c)\includegraphics[height=6cm]{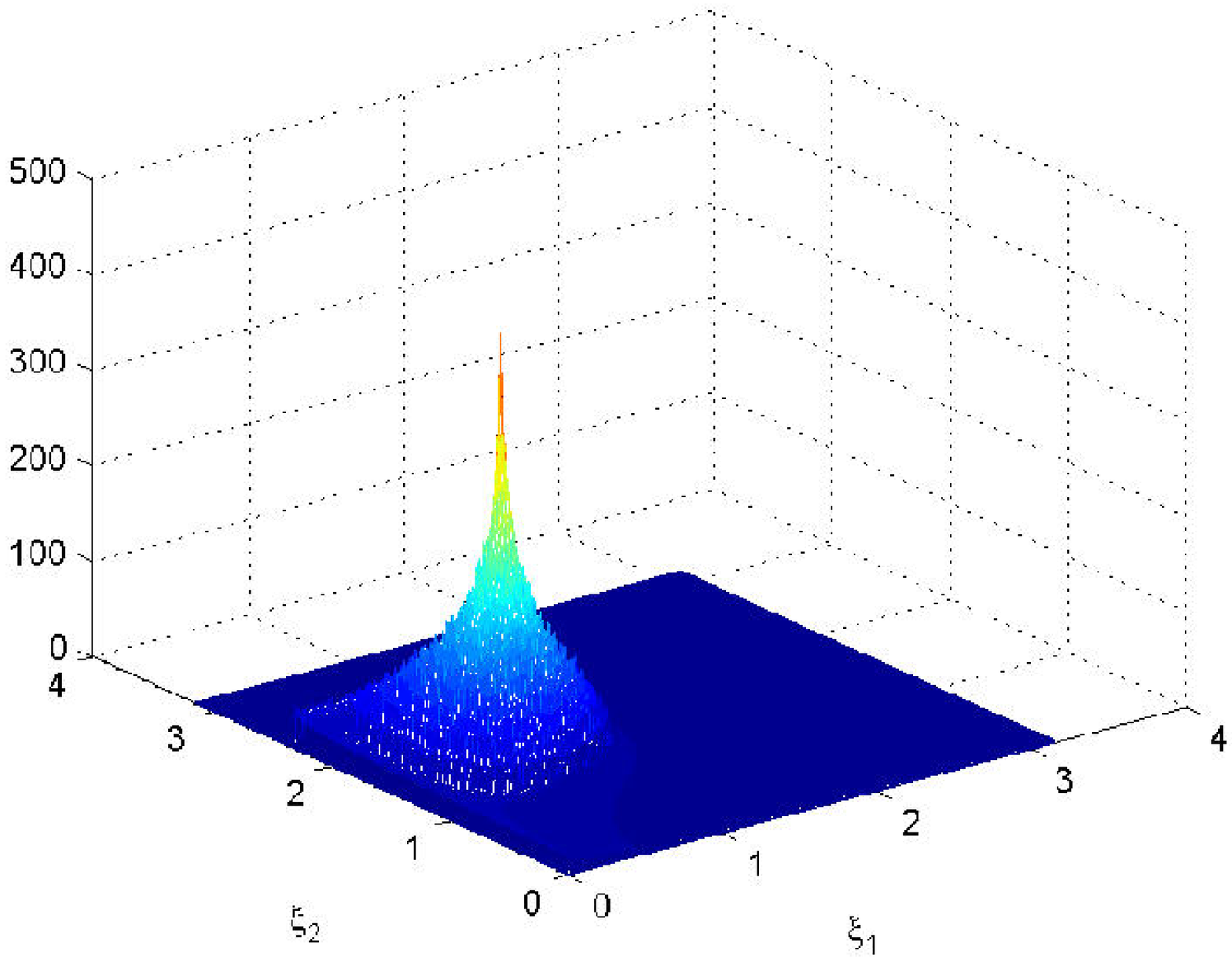}
\subfigure(d)\includegraphics[height=6cm]{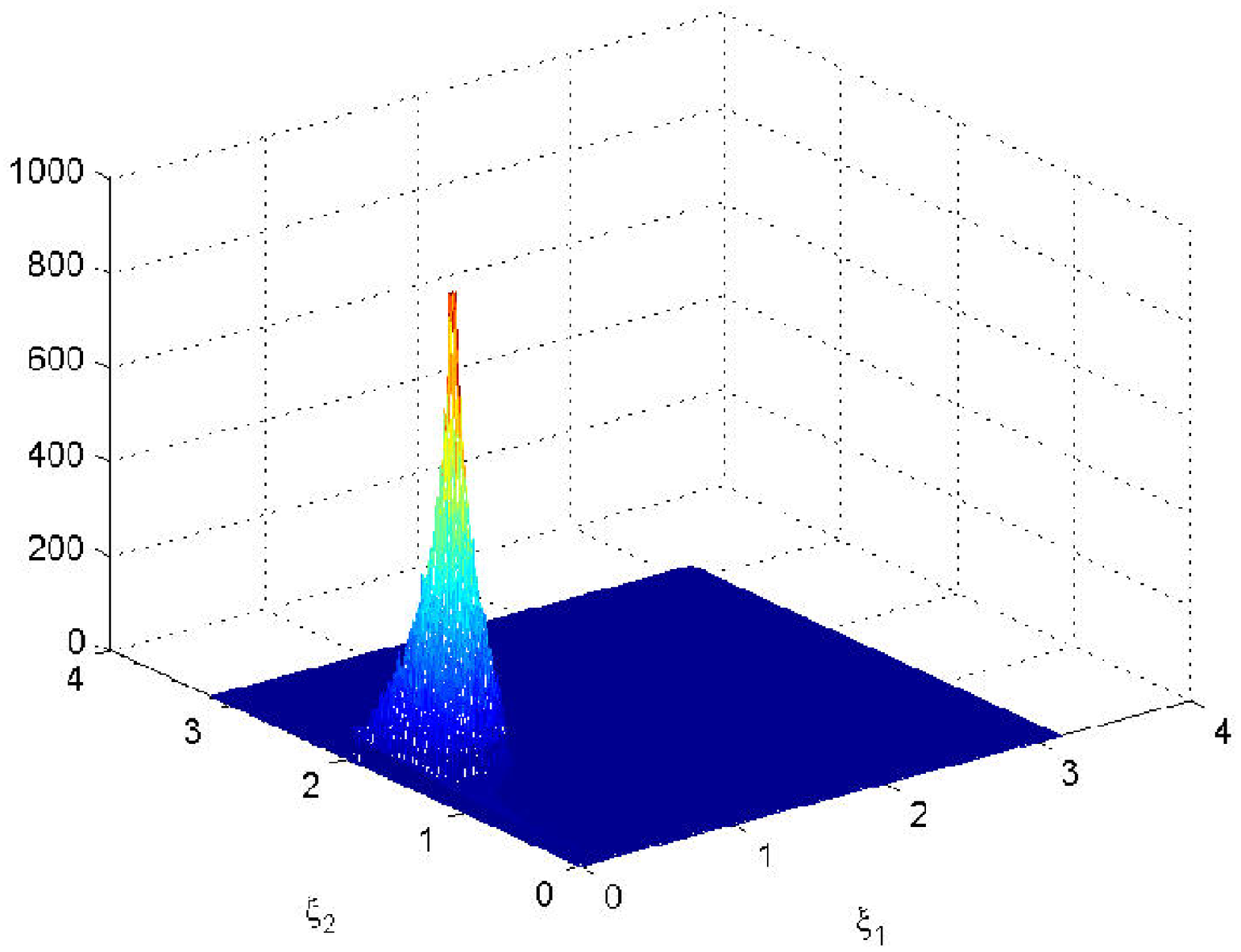}
\caption{First exit time for the slow system \eqref{reduced1} and
\eqref{reduced2} with $\varepsilon=0.05$, $a=0.7$ and $\sigma=0.01$:
(a) $V=0$, (b) $V=0.1$, (c) $V=0.5$ and (d) $V=0.65$. } \label{first
exit time}
\end{figure}

\begin{figure}
\subfigure(a)\includegraphics[height=6cm]{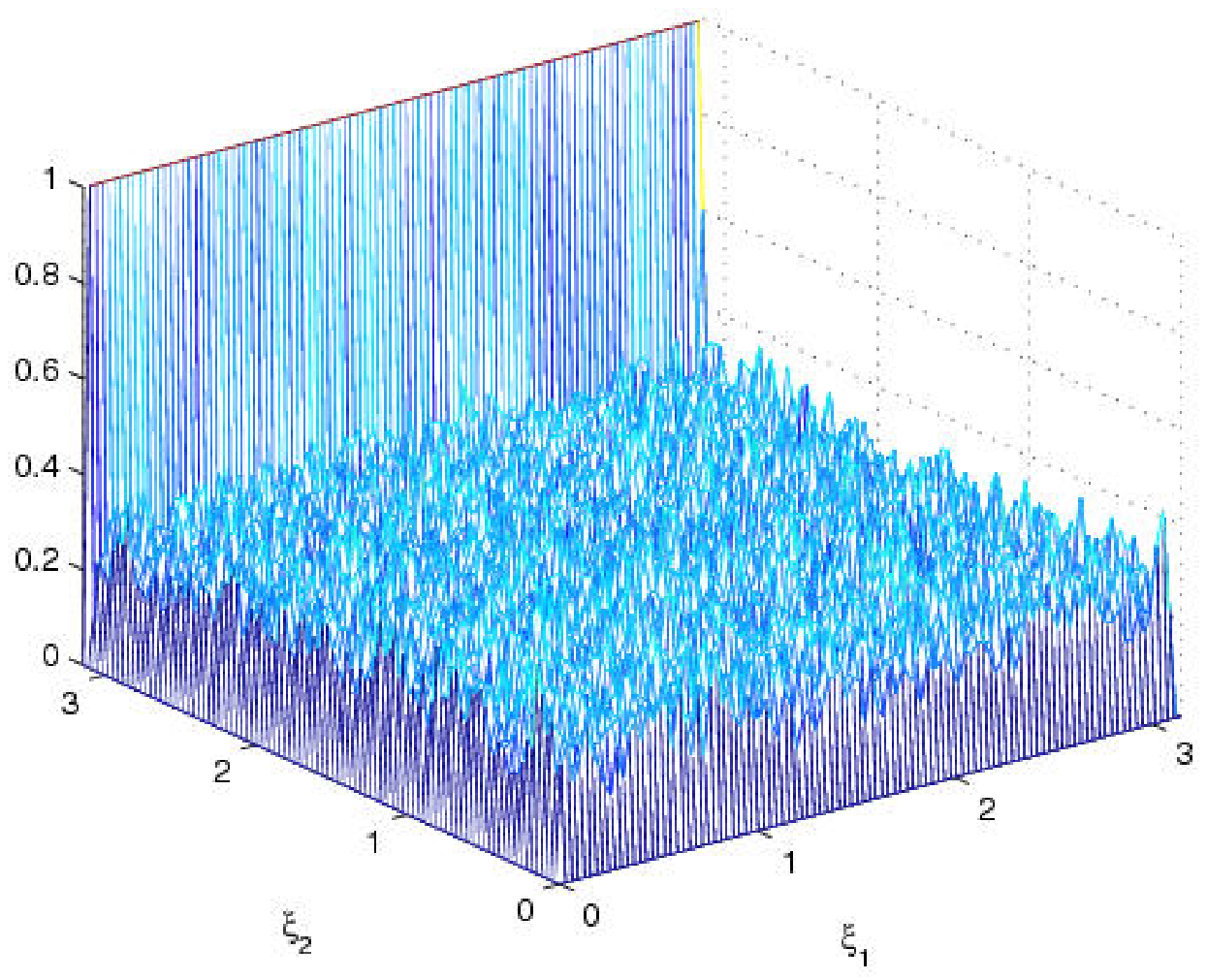}
\subfigure(b)\includegraphics[height=6cm]{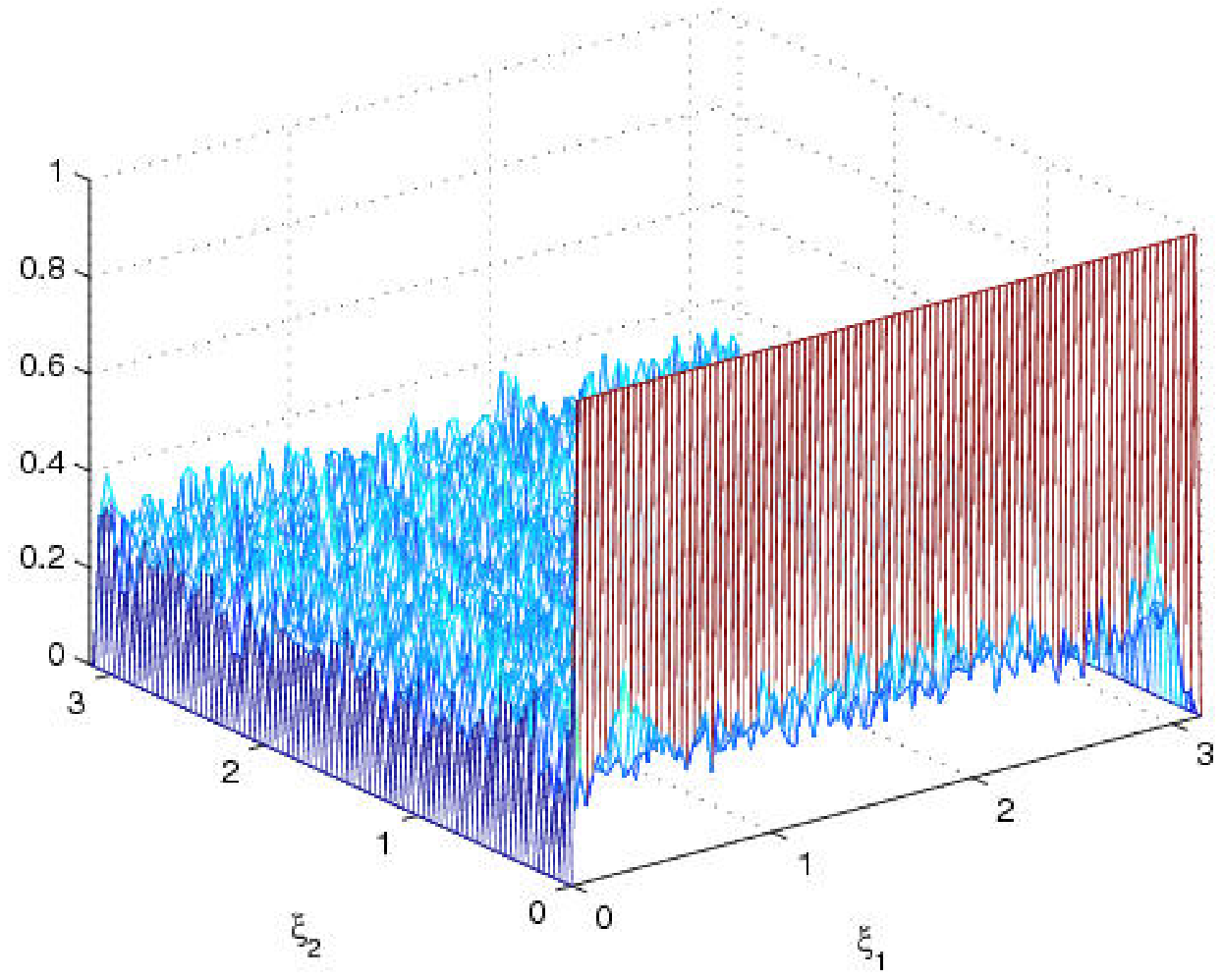}\\
\subfigure(c)\includegraphics[height=6cm]{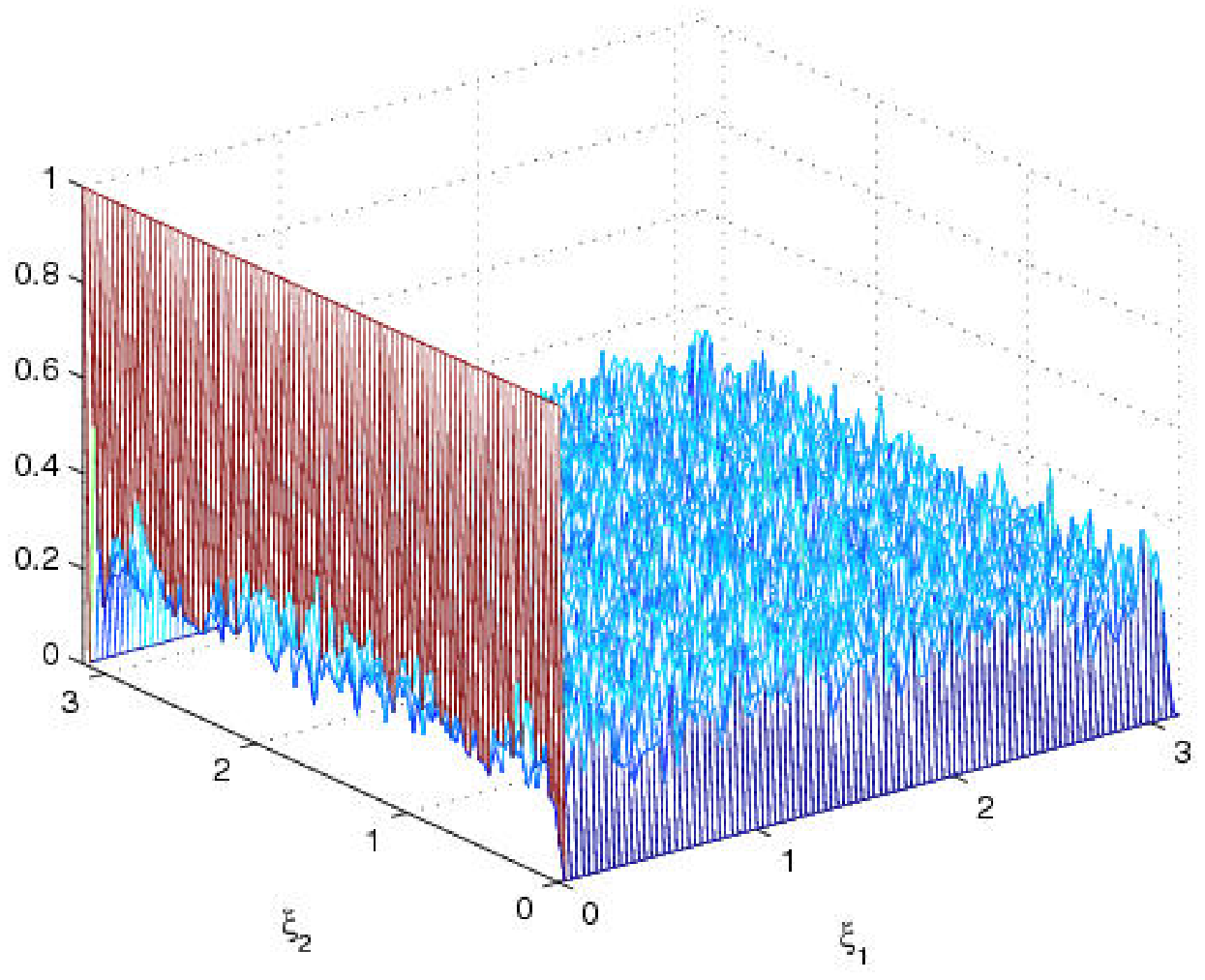}
\subfigure(d)\includegraphics[height=6cm]{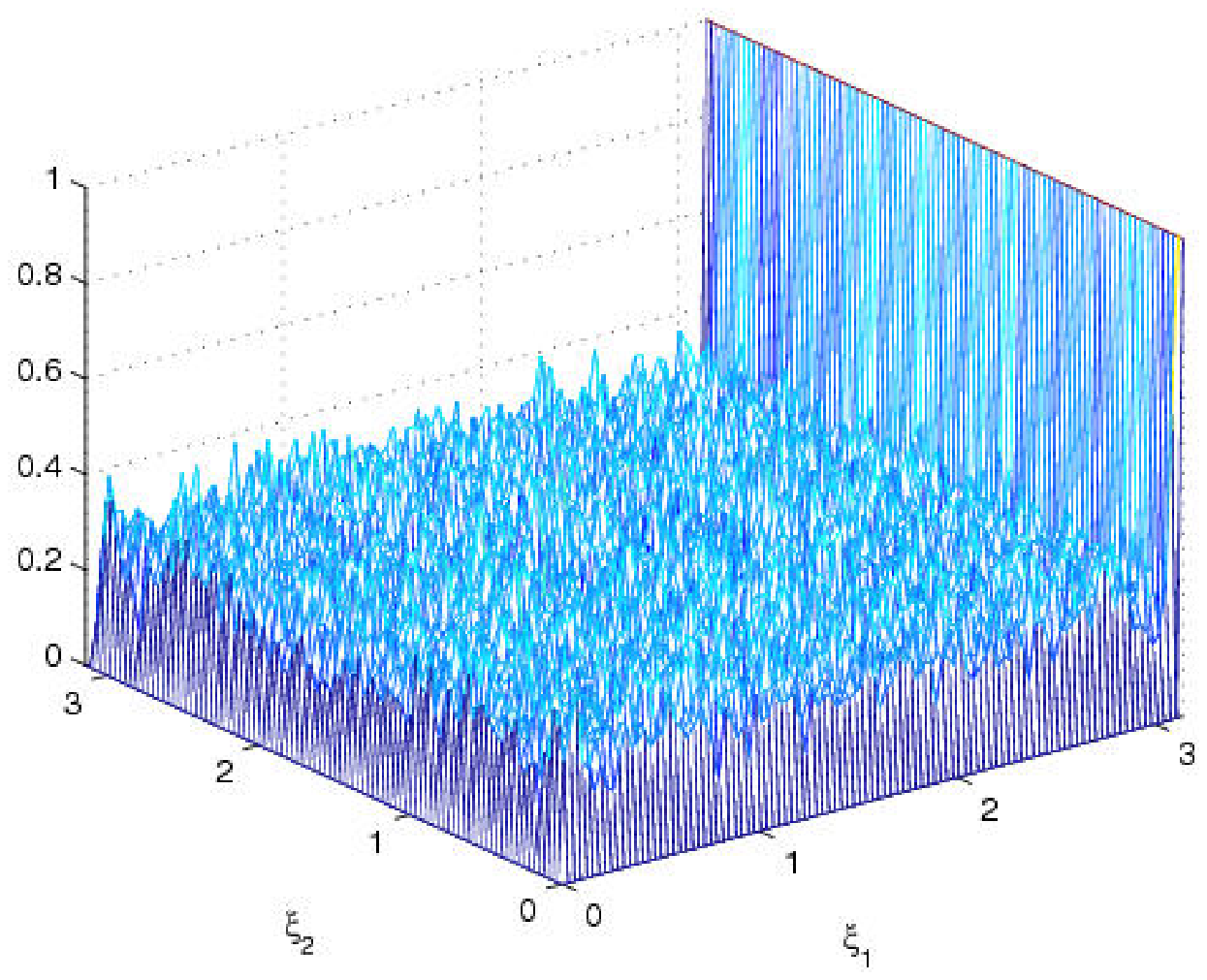}
\caption{Escape probability for the slow system \eqref{reduced1} and
\eqref{reduced2} with $V=0$, $\varepsilon=0.05$, $a=0.7$ and
$\sigma=0.01$: (a) escape through  $\xi_2=\pi$ (settling direction
or physical bottom boundary), (b) escape through $\xi_2=0$ (physical top boundary), (c)
escape through $\xi_1=0$ (left boundary), and (d) escape through
$\xi_1=\pi$ (right boundary).} \label{ep-w0}
\end{figure}

\begin{figure}
\subfigure(a)\includegraphics[height=6cm]{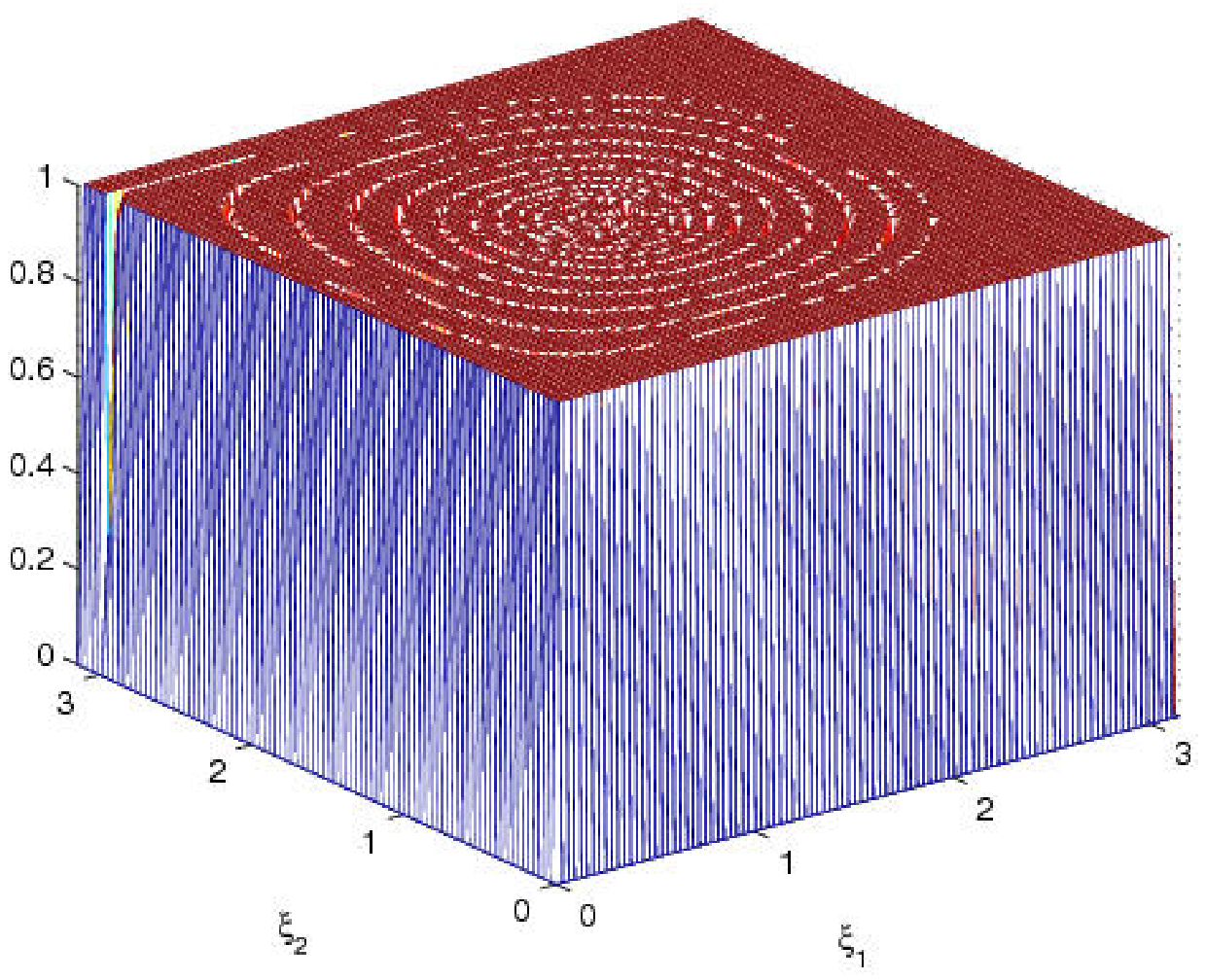}
\subfigure(b)\includegraphics[height=6cm]{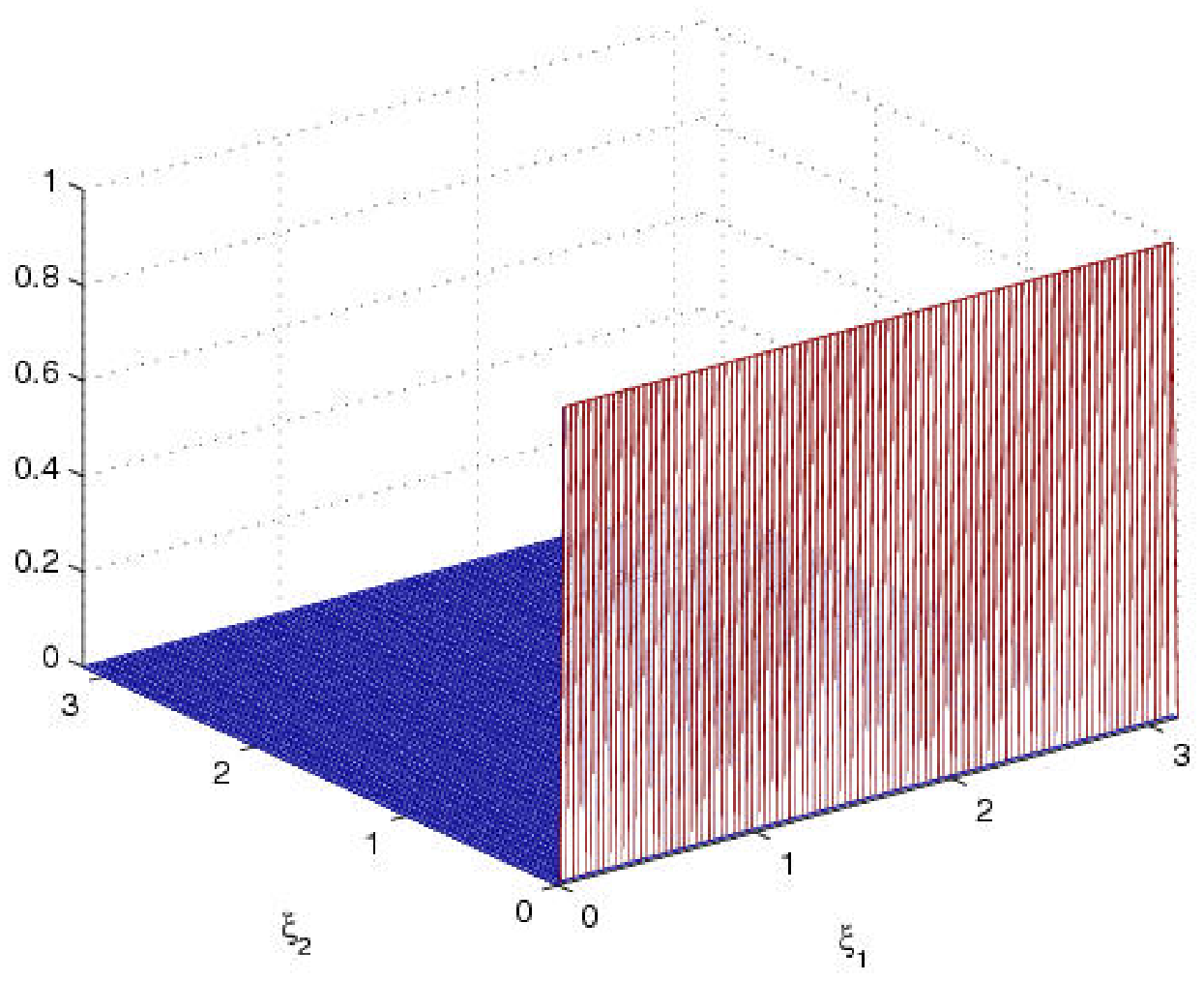}\\
\subfigure(c)\includegraphics[height=6cm]{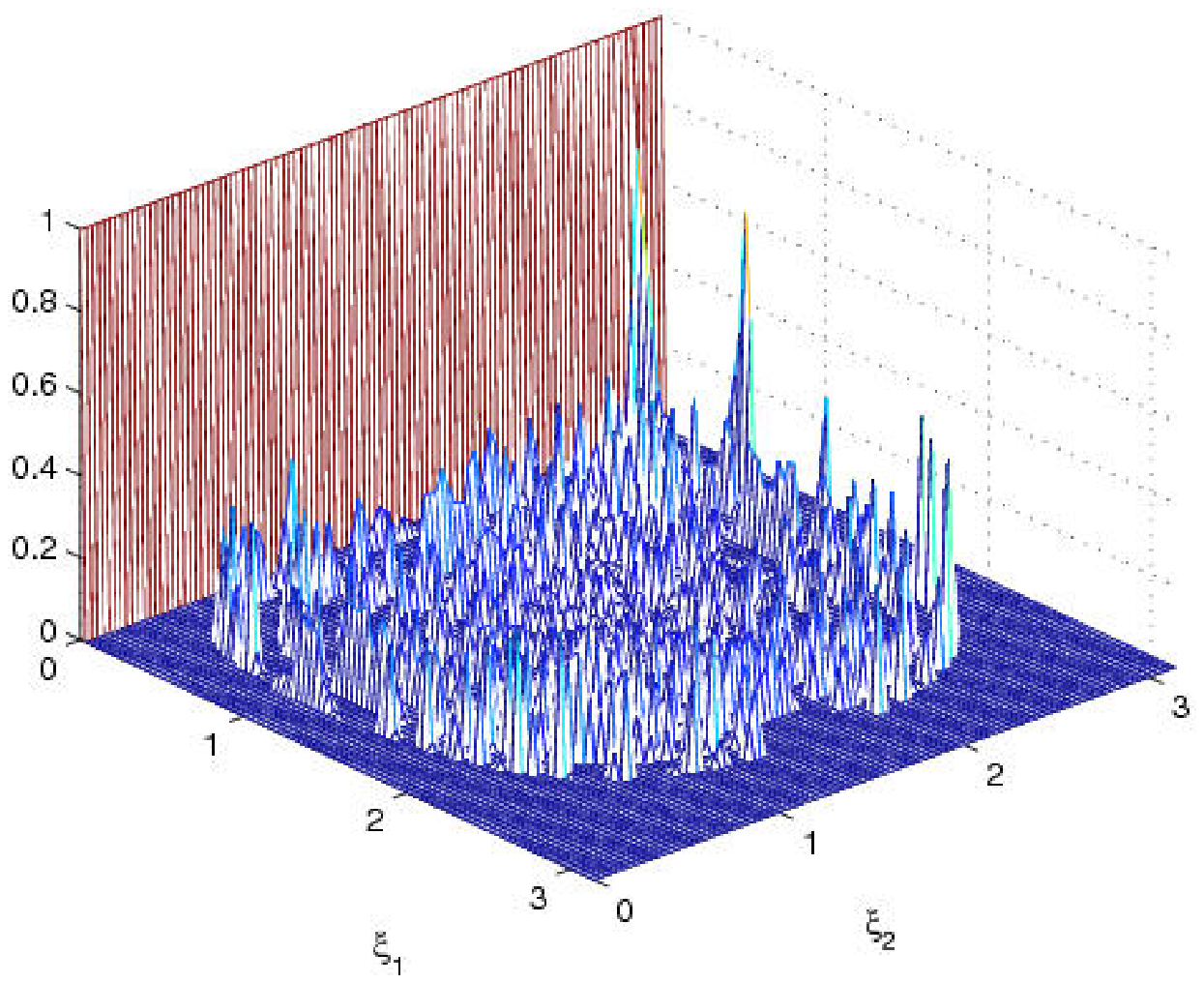}
\subfigure(d)\includegraphics[height=6cm]{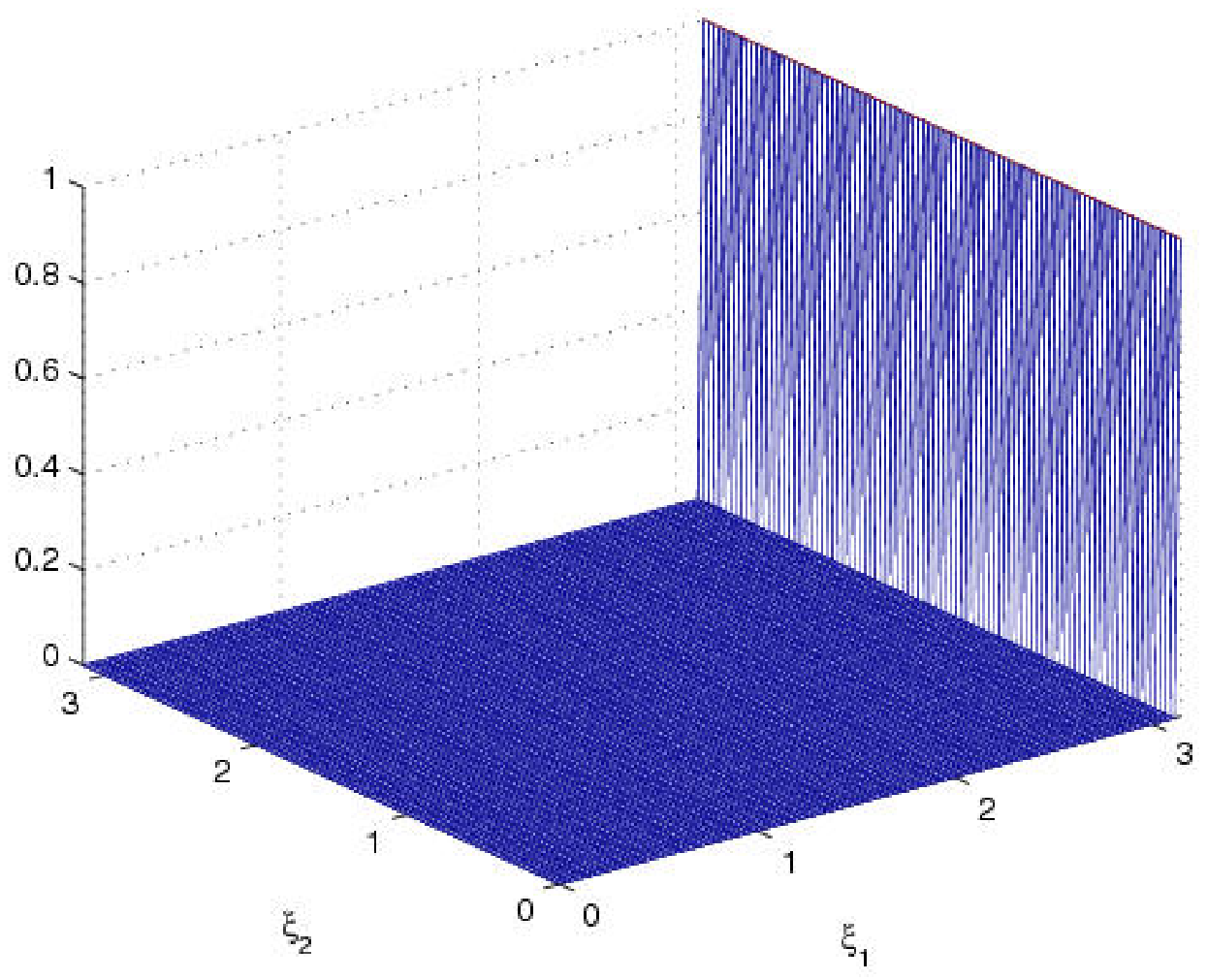}
\caption{Escape probability for the slow system \eqref{reduced1} and
\eqref{reduced2} with $V=0.1$, $\varepsilon=0.05$, $a=0.7$ and
$\sigma=0.01$: (a) Escape through $\xi_2=\pi$ (settling direction or
physical bottom boundary),  (b) escape through $\xi_2=0$ (physical top boundary), (c)
escape through $\xi_1=0$ (left boundary), and (d) escape through
$\xi_1=\pi$ (right boundary).} \label{ep-w01}
\end{figure}

\begin{figure}
\subfigure(a1)\includegraphics[height=6cm]{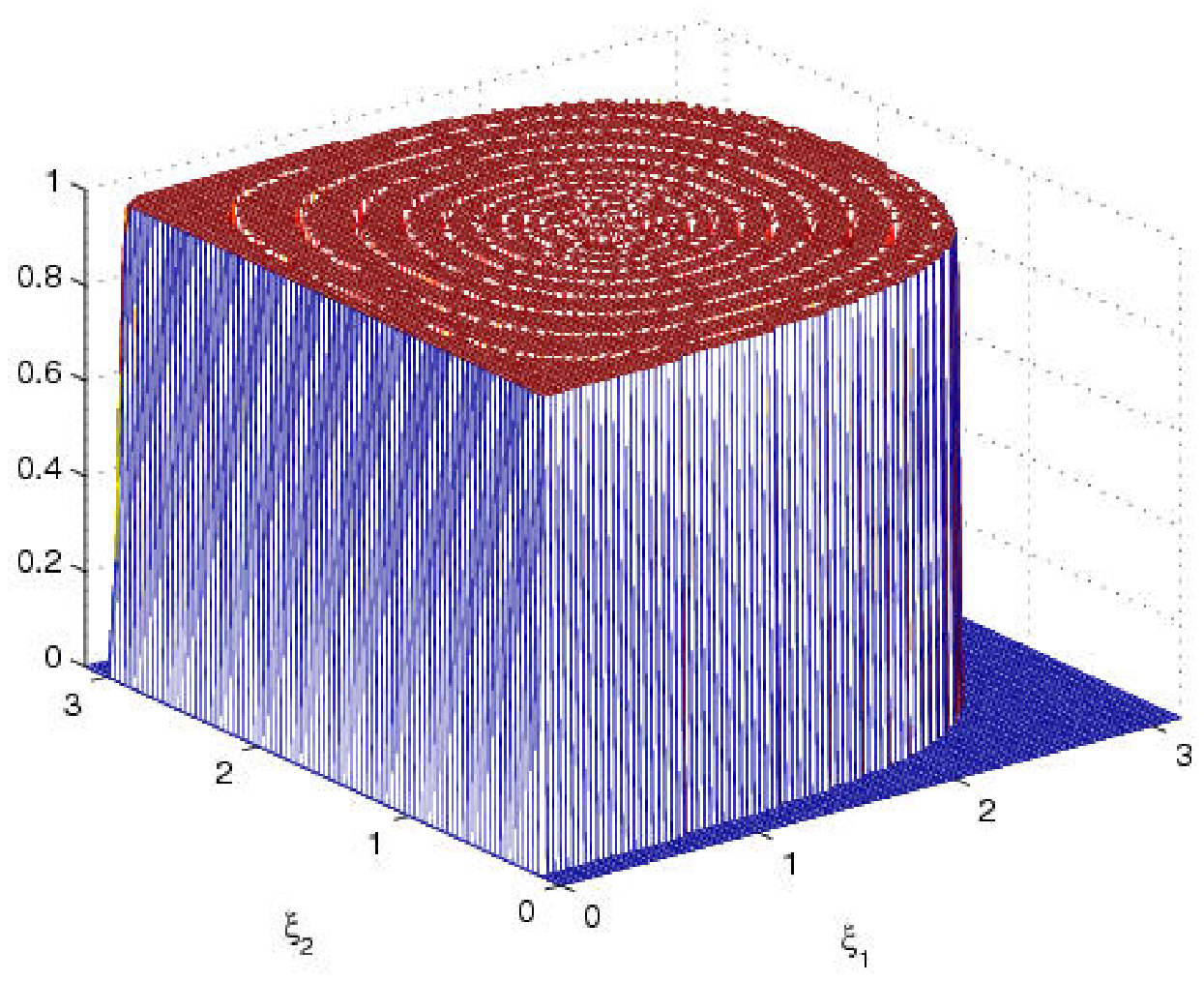}
\subfigure(a2)\includegraphics[height=6cm]{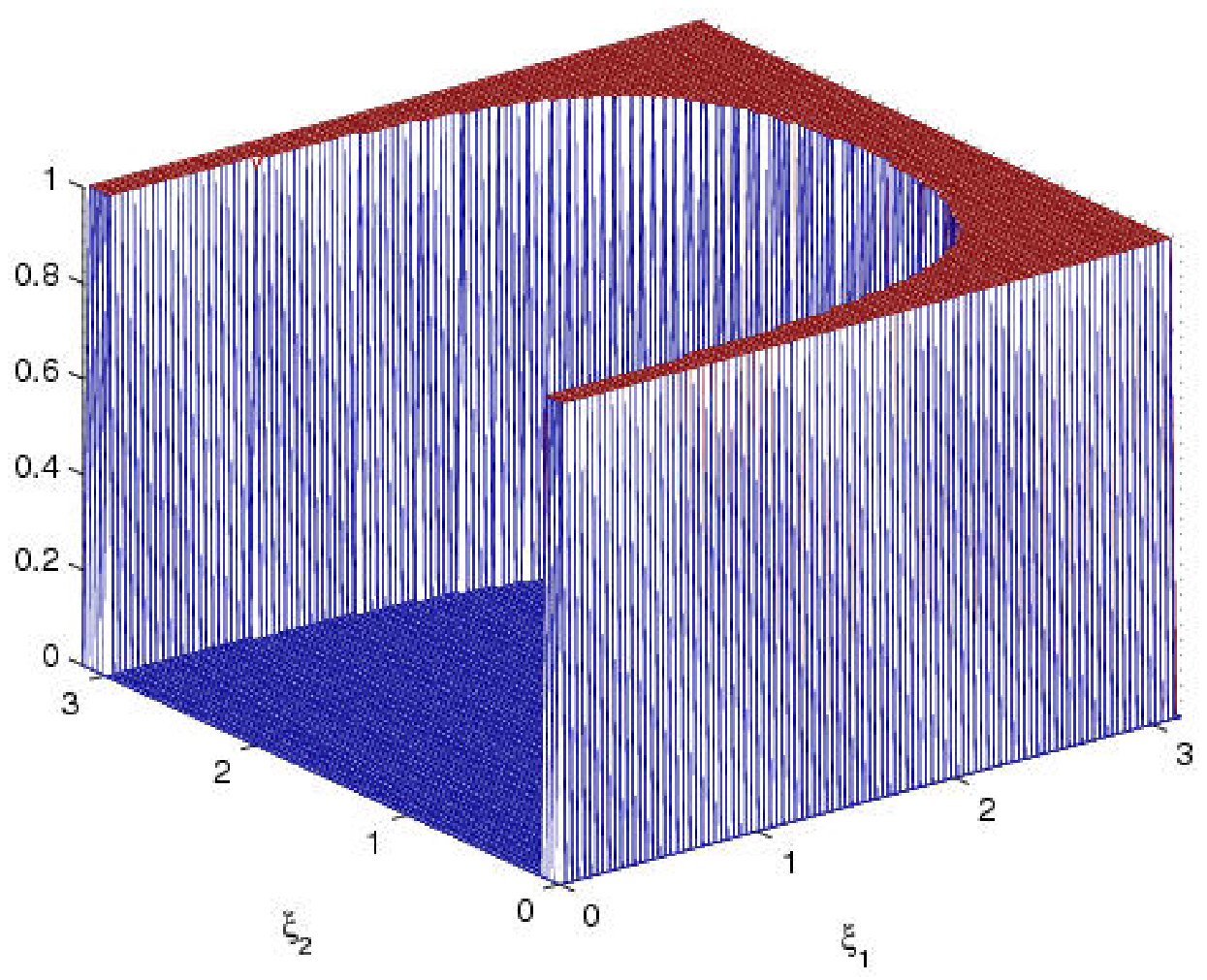}\\
\subfigure(c1)\includegraphics[height=6cm]{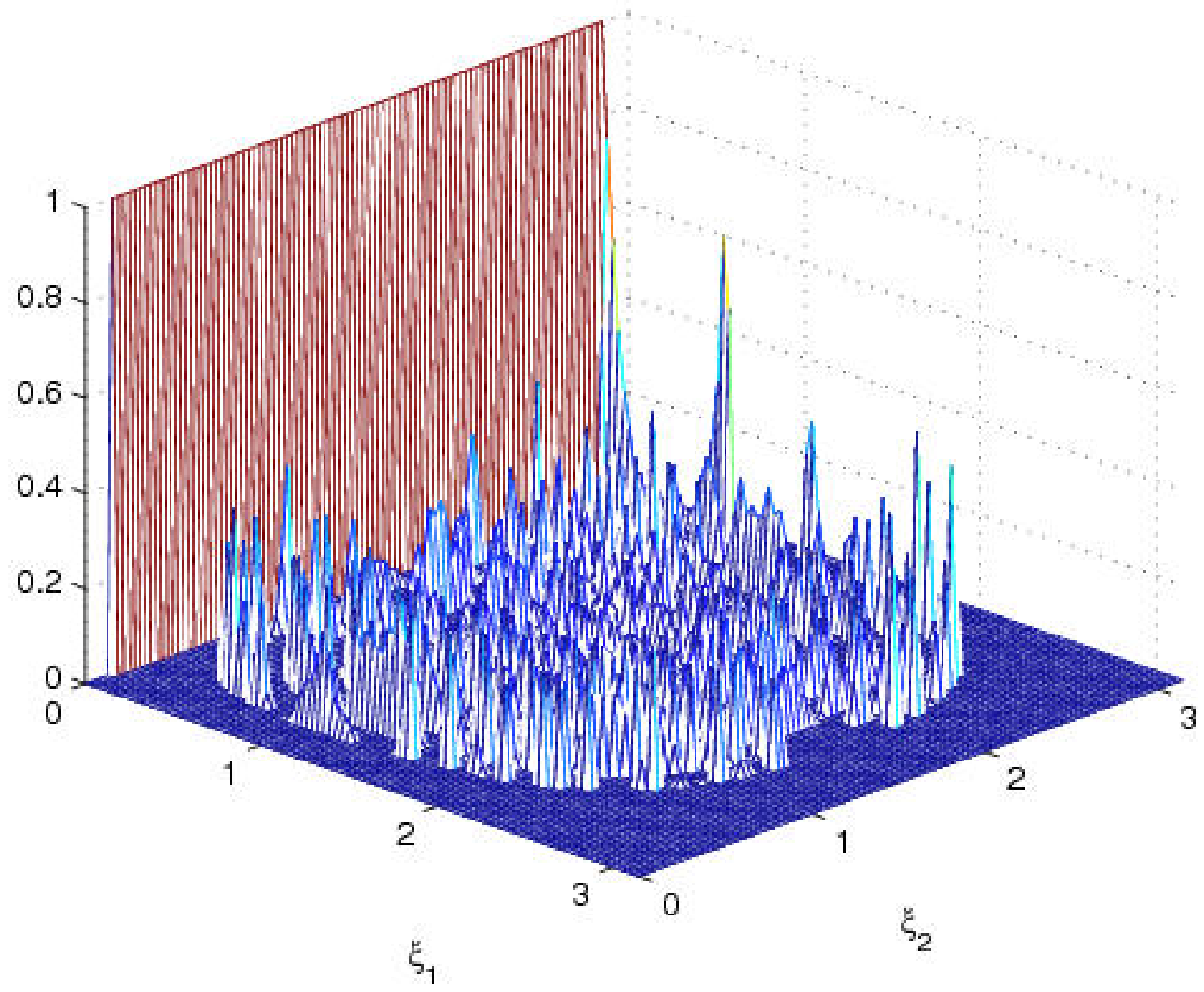}
\subfigure(c2)\includegraphics[height=6cm]{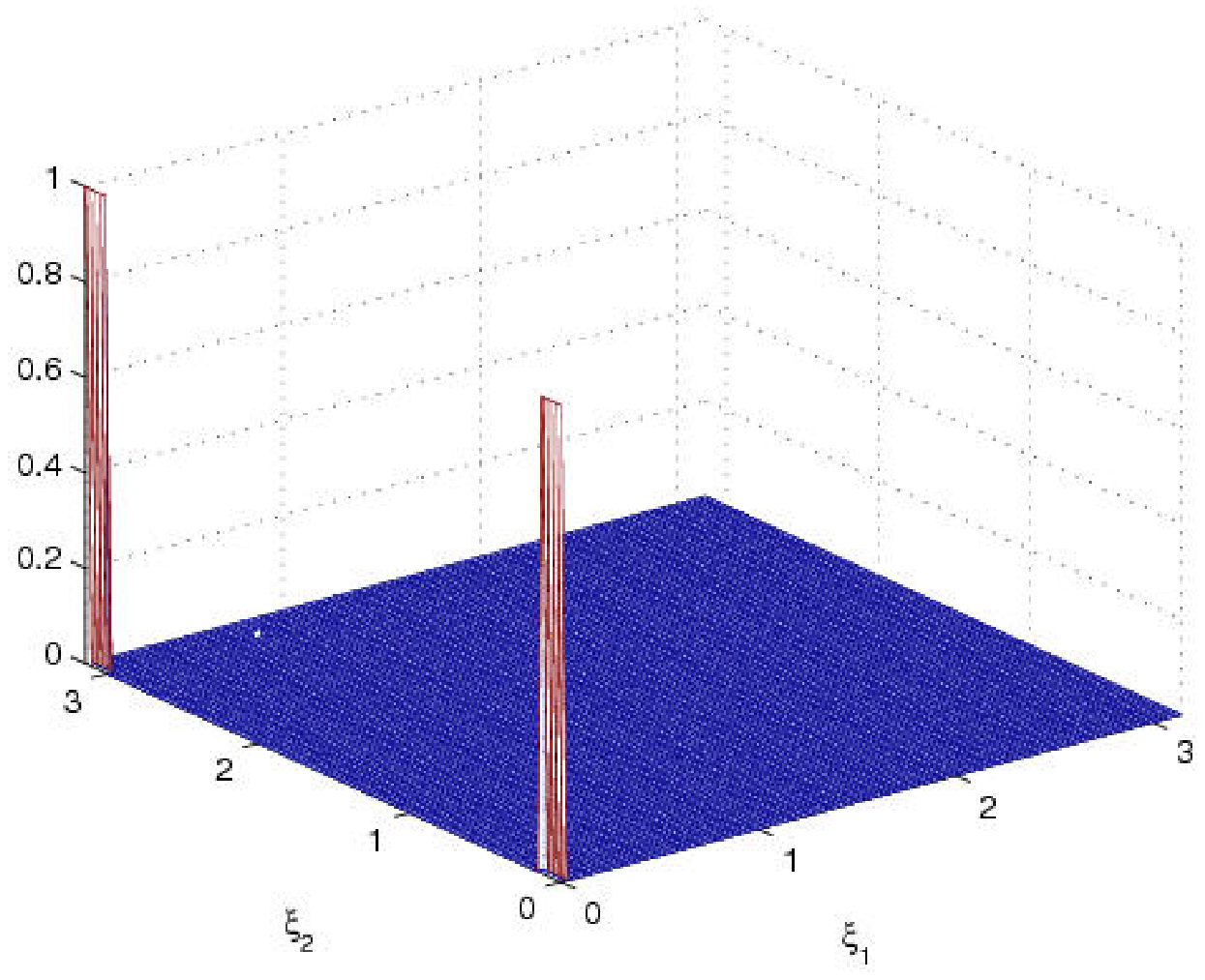}
\caption{Escape probability for the slow system \eqref{reduced1} and
\eqref{reduced2} with $V=0.1$, $\varepsilon=0.05$, $a=0.7$ and
$\sigma=0.01$: (a1) and (a2) are Figure \ref{ep-w01}(a) splitting at
the deterministic heteroclinic orbit $a\sin\xi_1\sin\xi_2-V \xi_1=0$;  (c1) and (c2) are Figure
\ref{ep-w01}(c) splitting at the same heteroclinic orbit.}
\label{ep-heter-w01}
\end{figure}

\begin{figure}\center
\includegraphics[height=7cm,width=8cm]{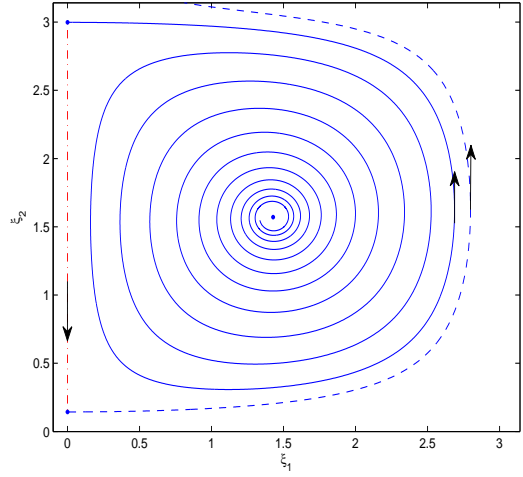}
\caption{Deterministic stable manifold $W^s(0, \pi-\sin^{-1}(\frac{V}{a}))$ (Blue solid curve) and unstable manifold $W^u(0,
\sin^{-1}(\frac{V}{a}))$ (Blue dashed curve) for the slow system
\eqref{reduced1} and \eqref{reduced2}: $V=0.1$,
$\varepsilon=0.05$, $a=0.7$,  and $\sigma=0$ (no noise). Two equilibrium points are
$(0, \pi-\sin^{-1}(\frac{V}{a}))$ and $(0,
\sin^{-1}(\frac{V}{a}))$.}
\label{stable-unstable manifolds v01}
\end{figure}

\begin{figure}\center
\includegraphics[height=8cm,width=12cm]{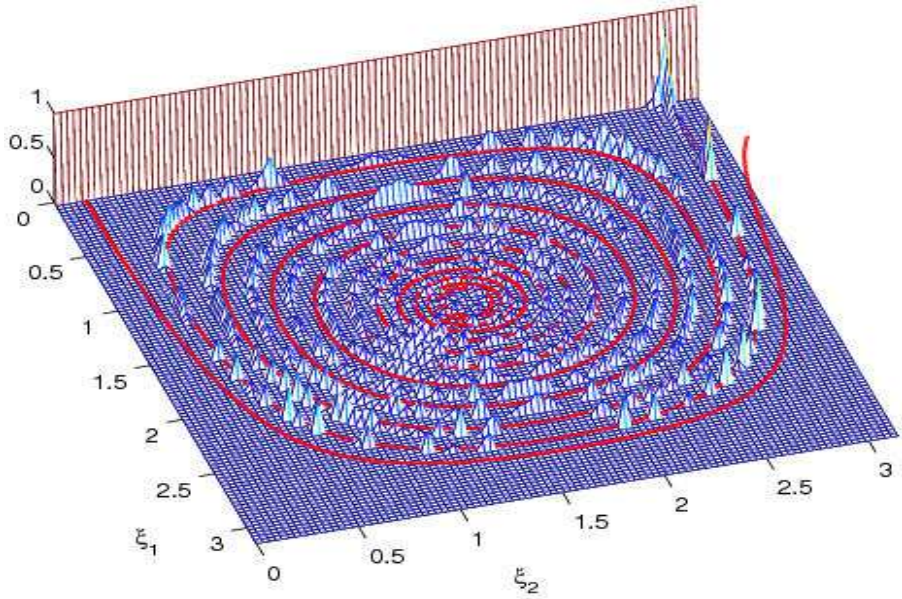}
\caption{Escape probability (Light blue color) for particles through the left boundary  for the slow system
\eqref{reduced1} and \eqref{reduced2} with  $V=0.1$,
$\varepsilon=0.05$, $a=0.7$ and $\sigma=0.01$ (with noise), together with deterministic stable manifold $W^s(0,
\pi-\sin^{-1}(\frac{V}{a}))$ and
unstable manifold $W^u(0, \sin^{-1}(\frac{V}{a}))$ (Red curves).}
\label{plot3 + mesh v01}
\end{figure}

\begin{figure}\center
\includegraphics[height=8cm,width=12cm]{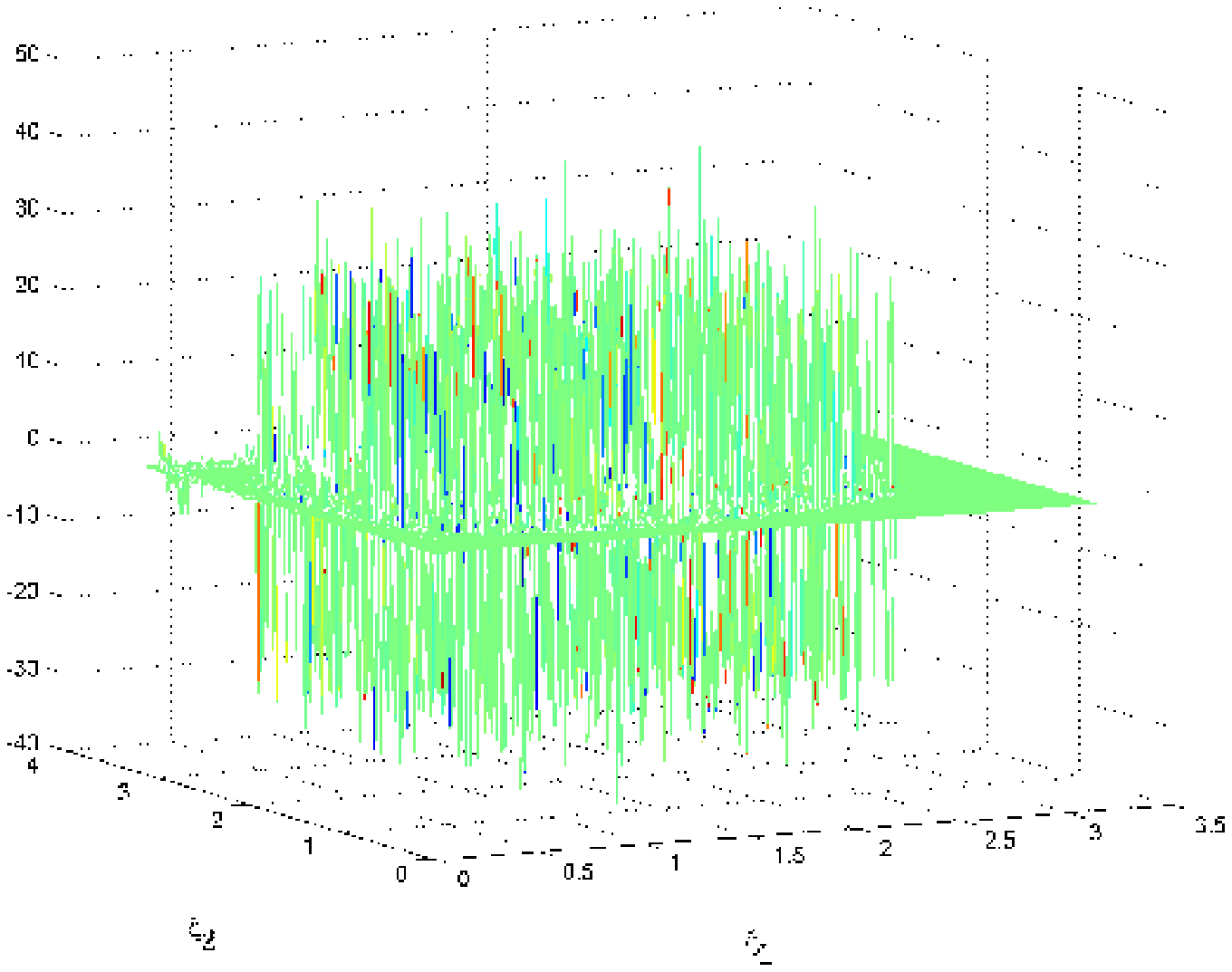}
\caption{Difference between the particle settling times for the deterministic case ($\sigma=0$) and the random case ($\sigma=0.01$)   of the slow system
\eqref{reduced1} and \eqref{reduced2}:  $V=0.1$,
$\varepsilon=0.05$ and $a=0.7$.}
\label{difference}
\end{figure}

\subsection{Deterministic Case}
We recall the results, by one of the present authors, of the deterministic motion of aerosol particles in a cellular flow \cite{Rubin}.

When the settling velocity in still fluid is zero, i.e.,  $V=0$ (and
also $\sigma=0$), particles are trapped in the cell either in
circular motion (with no inertial, $\varepsilon=0$) or spiralling
motion (with inertial, $\varepsilon>0$), as shown in Figure \ref{w0}
(top) and (bottom), respectively.

When the settling
velocity in still fluid is non-zero, i.e.,  $V>0$, in the   case  with inertial absent
($\varepsilon=0$) and noise absent ($\sigma=0$), the particles in the area surrounded
by the heteroclinic orbit $a \sin\xi_1\sin\xi_2-V \xi_1=0$ connecting the equilibrium points
    $(0, \sin^{-1}(V/a))$ and $(0, \pi-\sin^{-1}(V/a))$, are
trapped inside it,   with the equilibrium point $(\cos^{-1}(V/a),
\pi/2)$ as a center. But the particles in the remaining area settle
to the bottom of the fluid. With an arbitrarily small inertial
effect ($0<\varepsilon\ll1$ and also $\sigma=0$), the heteroclinic
orbit breaks and it leads to the settling of almost all particles,
with the equilibrium point $(\cos^{-1}(V/a)), \pi/2)$ becoming an
unstable spiral. Figure \ref{stable-unstable manifolds v01}  is the
stable manifold $W^s(0, \pi-\sin^{-1}(V/a))$  (Blue solid curve) and
unstable manifold $W^u(0, \sin^{-1}(V/a))$ (Blue dashed curve) when
inertial presents  ($\varepsilon=0.05$ and also $\sigma=0$).


\subsection{Stochastic Case}

For zero settling velocity in still fluid ($V=0$),   we note that all particles are trapped in a fluid cell $D=(0, \pi) \times (0, \pi)$  when noise is absent. Figure \ref{w0} shows  the  particle orbits in this case.

\bigskip

For non-zero settling velocity in still fluid ($V \neq 0$),
when noise is absent ($\sigma =0$ ), all particles settle to the fluid bottom; see
  Figure \ref{w03} (top, middle).
But when noise is present ($\sigma \neq 0$), some particles exit the
cell not only by settling. Figure \ref{w03} (bottom)  shows that,
with small noise ($\sigma=0.01$), some particles  indeed     exit
the cell from the vertical side boundary $\xi_1=0$.


In fact, when noise is present, all particles will exit, no matter
the settling velocity in still fluid  $V$ is zero or non-zero.
Figure \ref{first exit time} indicates that with noise, particles
will all exit from a fluid cell in  finite time, almost surely. In
the following we only consider the case with noise.



\bigskip

Figures \ref{ep-w0} --- \ref{ep-w01} plot the escape probability
through four side boundaries, for zero or non-zero $V$ (settling
velocity in still fluid) values. When a particle reaches or is on a
side boundary, it is regarded as `having escaped through' that part
of the boundary. In other words, particles on a side boundary have
escape probability $1$ (you see this in these figures). When $V=0$,
the particles escape the cell through each of  the four side
boundaries with similar or equal likelihood (Figure \ref{ep-w0}), as
there is no preferred direction for particles due to zero settling
velocity $V=0$ (in still fluid). With non-zero $V$, particles almost
surely do not escape through the right side  boundary $\xi_1=\pi$.
 In fact, the inertial particles either  settle to the physical bottom or
exit from the left side boundary $\xi_1=0$.   Figure \ref{ep-w01}
displays the escape probability for  $V=0.1$, through each of the
four side boundaries of the fluid cell. Although most particles
settle (Figure \ref{ep-w01} (a)), some particles escape the fluid
cell through the left side boundary (Figure \ref{ep-w01} (c)). See
Figure \ref{ep-heter-w01} for a split view of this phenomenon.

To examine this phenomenon more carefully, we draw the stable
manifold $W^s(0, \pi-\sin^{-1}(V/a))$   and unstable manifold
$W^u(0, \sin^{-1}(V/a))$   for the deterministic system ($\sigma=0$)
in Figure \ref{stable-unstable manifolds v01}. As shown in Figure
\ref{plot3 + mesh v01}, inertial particles with significant
likelihood  of escaping through the left side boundary are near or
on the stable manifold $W^s(0, \pi-\sin^{-1}(V/a))$. In other words,
some (but not all) inertial particles  near or on this stable
manifold  are resistent to settling in the stochastic case. This
resistance is quantified by the escape probability for a particle to
get out of the fluid cell through the left side boundary. More
specifically, the difference between the inertial particle settling
times for deterministic case ($\sigma=0$) and a  random case
($\sigma=0.01$) is shown in Figure \ref{difference}. We observe that
the inertial particles near or on the stable manifold $W^s(0,
\pi-\sin^{-1}(V/a))$ could have either a longer or shorter settling
time, compared with the deterministic case. This indicates that the
noise could either delay or enhance the settling (although we do not
know the reason), and the stable manifold is an agent  facilitating
this behavior. However, the overall impact of noise appears to delay
the settling, as the averaged difference over the cell $(0, \pi)
\times (0, \pi)$ is $ -0.0133$ for noise intensity $\sigma=0.01$,
while this averaged value is $-0.1168$ for a stronger noise with
$\sigma=0.1$.

\subsection{Conclusions}

Let the settling velocity in still fluid be non-zero (i.e., $V \neq 0$).   \\
(i) In the
classical case (no inertial: $\varepsilon=0$ and no noise: $ \sigma=0$),  the particles surrounded inside a heteroclinic orbit are trapped inside it and all the
other particles settle to the bottom of this cellular fluid flow. \\
(ii) In the case with only small inertia influence
($0<\varepsilon \ll 1, \sigma=0$), the heteroclinic orbit breaks up to form a stable manifold $W^s$ and an unstable manifold $W^u$, the trapped
particles then settle, i.e., all   inertial particles settle. \\
(iii) However, when the noise is present ($0<\varepsilon \ll 1, \sigma \neq 0$), although most inertial particles still settle, some particles near or on the deterministic stable manifold $W^s$ escape the fluid cell  through the left side boundary, with non-negligible likelihood. Thus, inertial particle motions occur in two adjacent fluid cells in   random cases, but confine in single cells in the deterministic case.

In fact, noise could either delay settling for some particles or enhance settling for others, and the deterministic stable manifold is an agent to facilitate this phenomenon. Overall, noise appears to delay the settling in an averaged sense.


\end{document}